\documentclass[11pt]{amsart}
\usepackage{amsmath,amssymb,txfonts}
\usepackage{amssymb}
\usepackage{amsmath}
\usepackage{mathrsfs}
\usepackage{amsmath,amssymb}
\newtheorem{theorem}{Theorem}[section]
\newtheorem{lemma}[theorem]{Lemma}

\newtheorem{corollary}[theorem]{Corollary}
\theoremstyle{definition}
\newtheorem{definition}[theorem]{Definition}

\theoremstyle{remark}
\newtheorem{remark}[theorem]{Remark}

\def\B{\mbox{\rm{I}}\!\mbox{\rm{B}}}
\def\R{\mbox{\rm{I}}\!\mbox{\rm{R}}}

\def\<{\leq}             \def\>{\geq}

\numberwithin{equation}{section}

\usepackage{color}

\begin{document}

\title[Modified Navier-Stokes Equations in Critical Besov-Q Spaces]
{Global Mild Solutions of Modified Navier-Stokes Equations with Small Initial Data in Critical Besov-Q Spaces}

\author[Pengtao Li]{Pengtao\ Li}
\address{College of Mathematics, Qingdao University, Qingdao, Shandong 515063, China}

\email{ptli@qdu.edu.cn}

\author[Jie Xiao]{Jie Xiao}
\address{Department of Mathematics and Statistics, Memorial
University, St. John's, NL A1C, 5S7, Canada }
\email{jxiao@mun.ca}

\author[Qixiang Yang]{Qixiang Yang}
\address{School of Mathematics and Statics, Wuhan University, Wuhan, 430072, China.}
\email{qxyang@whu.edu.cn}


\thanks{PTL's research was supported by: NSFC No. 11171203 and No.11201280;
New Teacher's Fund for Doctor Stations, Ministry of Education No.20114402120003;
Guangdong Natural Science Foundation S2011040004131; Foundation for
Distinguished Young Talents in Higher Education of Guangdong, China,
LYM11063. JX was supported by NSERC of Canada and URP of Memorial University, Canada}


\date{}

\dedicatory{}


\begin{abstract}
 This paper is devoted to establishing the global existence and uniqueness of a mild solution of the
 modified Navier-Stokes equations with a small initial data in the critical Besov-Q
space.
\end{abstract}

\maketitle

\tableofcontents \pagenumbering{arabic}

 \vspace{0.1in}

\section{Statement of the main theorem} 
\label{intro}

For $\beta>1/2$,  the Cauchy problem of the modified
Navier-Stokes equations on the half-space $\mathbb{R}^{1+n}_{+}=
(0,\infty) \times \mathbb{R}^{n}, n\geq 2,$ is to decide the existence of a solution $u$ to:
\begin{equation}\label{eqn:ns}
\left\{\begin{array}{ll} \frac{\partial u} {\partial t}
+(-\Delta)^{\beta} u + u \cdot \nabla u -\nabla p=0,
& \mbox{ in } \mathbb{R}^{1+n}_{+}; \\
\nabla \cdot u=0,
& \mbox{ in } \mathbb{R}^{1+n}_{+}; \\
u|_{t=0}= a, & \mbox{ in } \mathbb{R}^{n},
\end{array}
\right.
\end{equation}
where $(-\Delta)^{\beta}$ represents the $\beta$-order Laplace operator
defined by the Fourier transform in the space variable:
$$
\widehat{(-\Delta)^{\beta}u}(\cdot,\xi)= |\xi|^{2\beta} \hat{u}(\cdot,\xi).
$$

Here, it is appropriate to point out that (\ref{eqn:ns}) is a generalization of the classical Navier-Stokes system and two-dimensional  quasi-geostrophic equation which have continued to attract attention extensively, and that the dissipation $(-\Delta)^{\beta}u$ still retains the physical meaning of the nonlinearity $u\cdot\nabla u+\nabla p$ and the divergence-free condition $\nabla\cdot u=0$.

Upon letting $R_{j}, j=1,2,\cdots n$, be the Riesz transforms,
writing
\begin{equation}\nonumber
\begin{cases}
\mathbb{P}= \{\delta_{l,l'}+ R_{l}R_{l'}\}, l,l'=1,\cdots,n;\\
\mathbb{P}\nabla (u\otimes u)= \sum\limits_{l}
\frac{\partial}{\partial x_{l}} (u_{l}u) -
\sum\limits_{l} \sum\limits_{l'} R_{l}R_{l'} \nabla (u_{l} u_{l'});\\
\widehat{e^{-t(-\Delta)^{\beta}}f}(\xi) =
e^{-t|\xi|^{2\beta}}\hat{f}(\xi),
\end{cases}
\end{equation}
and using $\nabla \cdot u=0$, we can see that a solution of the above
Cauchy problem is then obtained via the integral equation:
\begin{equation}\label{eqn:mildsolution}
\begin{cases}
u(t,x)= e^{-t (-\Delta)^{\beta}} a(x) - B(u,u)(t,x);\\
B(u,u)(t,x)\equiv\int^{t}_{0} e^{-(t-s)(-\Delta)^{\beta}}
\mathbb{P}\nabla (u\otimes u) ds,
\end{cases}
\end{equation}
which can be solved by a fixed-point method whenever the convergence
is suitably defined in a function space. Solutions of
(\ref{eqn:mildsolution}) are called mild solutions of
(\ref{eqn:ns}). The notion of such a mild solution was pioneered by
Kato-Fujita \cite{KF} in 1960s. During the latest decades, many
important results about the mild solutions to (\ref{eqn:ns}) have been
established; see for example, Cannone \cite{C1, C2},
Germin-Pavlovic-Staffilani \cite{GPS}, Giga-Miyakawa \cite{GM},
Kato \cite{Kat},  Koch-Tataru \cite{KT},
Wu \cite{W1,W2,W3,W4}, and their references including Kato-Ponce \cite{KP} and Taylor \cite{Ta}.

The main purpose of this paper is to establish the following global
existence and uniqueness of a mild solution to (\ref{eqn:ns}) with a small
initial data in the critical Besov-Q space.

\begin{theorem}\label{mthmain} Given
$$
\begin{cases}
\beta>\frac{1}{2};\\
1< p, q<\infty;\\
\gamma_{1}=\gamma_{2}-2\beta+1;\\
m>\max \{p,\frac{n}{2\beta}\};\\
0<m'<\min\{1,\frac{p}{2\beta}\}.
\end{cases}
$$
If the index $(\beta, p,\gamma_2)$ obeys
\begin{equation}\nonumber
1< p\leq2\ \ \&\ \ \frac{2\beta-2}{p}<\gamma_{2}\leq\frac{n}{p}
\end{equation}
or
\begin{equation}\nonumber
 2<p<\infty\ \ \&\ \ \beta-1<\gamma_{2}\leq\frac{n}{p},
\end{equation}
then (\ref{eqn:ns}) has a unique global mild solution in
$(\B^{\gamma_{1}, \gamma_{2}}_{p, q, m, m'})^{n}$ for any initial
data $a$ with $\|a\|_{(\dot{B}^{\gamma_{1},\gamma_{2}}_{p,q})^{n}}$
being small.
Here the symbols
$\dot{B}^{\gamma_{1},\gamma_{2}}_{p,q}$, and $\B^{\gamma_{1},
\gamma_{2}}_{p, q, m, m'}$ stand for the so-called Besov-Q spaces
and their induced tent spaces, and will be determined
properly in Sections \ref{sec3} and \ref{sec4}.

\end{theorem}

Neededless to say, our current work grows from the already-known
results. In \cite{Lio}, Lions proved the global existence of the
classical solutions of (\ref{eqn:ns}) when $\beta\geq\frac{5}{4}$ and $n=3$.
This existence result was extended to $\beta\geq
\frac{1}{2}+\frac{n}{4}$ by Wu \cite{W1}, and moreover, for the
important case $\beta<\frac{1}{2}+\frac{n}{4}$. Wu \cite{W2, W3}
established the global existence for (\ref{eqn:ns}) in the Besov
spaces $\dot{B}^{1+\frac{n}{p}-2\beta,q}_{p}(\mathbb{R}^{n})$ for
$1\leq q\leq \infty$ and for either $\frac{1}{2}<\beta$ and $p=2$ or
$\frac{1}{2}<\beta\leq1$ and $2<p<\infty$ and in
$\dot{B}^{r,\infty}_{2}(\mathbb{R}^{n})$ with
$r>\max\{1,1+\frac{n}{p}-2\beta\}$; see also \cite{W4} concerning
the corresponding regularity. Importantly, Koch-Tataru \cite{KT}
studied the global existence and uniqueness of (\ref{eqn:ns}) with $\beta=1$
via introducing $BMO^{-1}(\mathbb{R}^{n})$.
Extending Koch-Tataru's work \cite{KT}, Xiao \cite{X, X1} introduced
the Q-spaces $Q^{-1}_{0<\alpha<1}(\mathbb{R}^{n})$ to
investigate the global existence and uniqueness of the classical Navier-Stokes
system. The ideas of \cite{X} were developed by Li-Zhai \cite{LZ}
to study the global existence and uniqueness of (\ref{eqn:ns}) with small
data in a class of Q-type spaces $Q_{\alpha}^{\beta,-1}(\mathbb{R}^{n})$
under $\beta\in(\frac{1}{2}, 1)$. Recently, Lin-Yang \cite{LY} got
the global existence and uniqueness of (\ref{eqn:ns}) with initial
data being small in a diagonal Besov-Q space for $\beta\in(\frac{1}{2}, 1)$.

In fact, the above historical citations lead us to make a decisive
two-fold observation. On the one hand, thanks to that (\ref{eqn:ns})
is invariant under the scaling
\begin{equation}\nonumber
\begin{cases}
u_{\lambda}(t,x) = \lambda^{2\beta-1} u(\lambda^{2\beta}t, \lambda
x);\\
p_{\lambda}(t,x) = \lambda^{4\beta-2} p(\lambda^{2\beta}t, \lambda
x),
\end{cases}
\end{equation}
the initial data space $\dot{B}^{\gamma_{1},\gamma_{2}}_{p,q}$ is
critical for (\ref{eqn:ns}) in the sense that the space is invariant
under the scaling
\begin{equation}\label{eq2}
f_{\lambda}(x)=\lambda^{2\beta-1}f(\lambda x).
\end{equation}
A simple computation, along with letting $\beta=1$ in (\ref{eq2}),
indicates that the function spaces:
$$
\begin{cases}
\dot{L}^{2}_{\frac{n}{2}-1}(\mathbb{R}^{n})=\dot{B}^{-1+\frac{n}{2},2}_{2}(\mathbb{R}^{n});\\
L^{n}(\mathbb{R}^{n});\\
\dot{B}^{-1+\frac{n}{p},q}_{p}(\mathbb{R}^{n});\\
BMO^{-1}(\mathbb{R}^{n}),
\end{cases}
$$
are critical for (\ref{eqn:ns}) with $\beta=1$. Moreover,
(\ref{eq2}) under $\beta>1/2$ is valid for functions in the
homogeneous Besov spaces
$\dot{B}^{1+\frac{n}{2}-2\beta,1}_{2}(\mathbb{R}^{n})$ and
$\dot{B}^{1+\frac{n}{2}-2\beta,\infty}_{2}(\mathbb{R}^{n})$ attached
to (\ref{eqn:ns}). On the other hand, it is suitable to mention
the following relations:
$$
\begin{cases}
\dot{B}^{\gamma_{1},
\frac{n}{p}}_{p,q}=\dot{B}^{\gamma_{1},q}_{p}(\mathbb{R}^{n})\ \hbox{for}\ 1\leq p, q<\infty\ \&\ -\infty<\gamma_{1}<\infty;\\
\dot{B}^{1+\frac{n}{p}-2\beta,\frac{n}{p}}_{p_{0},q_{0}}\supseteq
\dot{B}^{1+\frac{n}{p}-2\beta,q}_{p}(\mathbb{R}^{n})\ \hbox{for}\ 1<p\leq p_{0}\ \&\ 1<q\leq q_{0}<\infty\ \&\ \beta>0;\\
\dot{B}^{\alpha-\beta+1,\alpha+\beta-1}_{2,2}=Q^{\beta}_{\alpha}(\mathbb{R}^{n})
\ \hbox{for}\  \alpha\in(0,1)\ \&\ \beta\in(1/2,1)\ \& \
\alpha+\beta-1\geq0.
\end{cases}
$$

In order to briefly describe the argument for Theorem \ref{mthmain},
we should point out that the function spaces used in \cite{KT,X,LZ} have a
common trait in the structure, i.e., these spaces can be seen as the
Q-spaces with $L^{2}$ norm, and the advantage of such spaces is that
Fourier transform plays an important role in estimating the bilinear
term on the corresponding solution spaces. Nevertheless, for the
global existence and uniqueness of a mild solution to (\ref{eqn:ns}) with a
small initial data in $\dot{B}^{\gamma_{1},\gamma_{2}}_{p,q}$, we
have to seek a new approach. Generally speaking, a mild solution of
(\ref{eqn:ns}) is obtained by using the following method. Assume
that the initial data belongs to
$\dot{B}^{\gamma_{1},\gamma_{2}}_{p,q}(\mathbb{R}^{n})$. Via the
iteration process:
$$
\begin{cases}
u^{(0)}(t,x)=e^{-t(-\Delta)^{\beta}} a(x);\\
u^{(j+1)}(t,x)=u^{(0)}(t,x)-B(u^{(j)},u^{(j)})(t,x)\quad\hbox{for}\quad
j=0,1,2,...,
\end{cases}
$$
we construct a contraction mapping on a space in
$\mathbb{R}^{1+n}_{+}$, denoted by $X(\mathbb{R}^{1+n}_{+})$. With
the initial data being small, the fixed point theorem implies that
there exists a unique mild solution of (\ref{eqn:ns}) in
$X(\mathbb{R}^{1+n}_{+})$. In this paper, we choose
$X(\mathbb{R}^{1+n}_{+})=\B^{\gamma_{1},\gamma_{2}}_{p, q, m, m'}$
associated with $\dot{B}^{\gamma_{1},\gamma_{2}}_{p, q}$. Owing to Theorem \ref{th1},
we know that if $f\in \dot{B}^{\gamma_{1},\gamma_{2}}_{p,q}$ then $
e^{-t(-\Delta)^{\beta}}f(x)\in X(\mathbb{R}^{1+n}_{+})$. Hence the construction of contraction mapping comes down to prove the following
assertion:
 the bilinear operator
$$
\begin{array}{rl}
B(u,v)=&\int^{t}_{0} e^{-(t-s)(-\Delta)^{\beta}} \mathbb{P}\nabla
(u\otimes v) ds\end{array}
$$
is bounded from $(X(\mathbb{R}^{1+n}_{+}))^n \times
(X(\mathbb{R}^{1+n}_{+}))^n $ to $(X(\mathbb{R}^{1+n}_{+}))^n $.

For this purpose, using multi-resolution analysis, we decompose
 $B_{l}(u,v)$ into several parts based on the relation between $t$ and $2^{-2j\beta}$, and expanse every part in terms of $\{\Phi^{\varepsilon}_{j,k}\}$. More importantly, Lemmas \ref{lem53} \& \ref{lem54} enable us to
 obtain an estimate from  $(\B^{\gamma_{1},\gamma_{2}}_{p,q, m,m'})^n
\times (\B^{\gamma_{1},\gamma_{2}}_{p,q, m,m'})^n$ to
$(\B^{\gamma_{1},\gamma_{2}}_{p,q, m,m'})^n$.

\begin{remark}
\item{\rm(i)} Our initial spaces in Theorem \ref{mthmain} include both
 $\dot{B}^{1+\frac{n}{p}-2\beta,q}_{p}(\mathbb{R}^{n})$ in Wu
\cite{W1,W2,W3,W4}, $Q_{\alpha}^{\beta,-1}(\mathbb{R}^{n})$ in Xiao
\cite{X} and Li-Zhai \cite{LZ}. Moreover, in \cite{LZ, LY}, the
scope of $\beta$ is $(\frac{1}{2}, 1)$. Our method is valid for
$\beta>\frac{1}{2}$.

\item{\rm(ii)} We point out that $\dot{B}^{\gamma_{1}, \gamma_{2}}_{p,q}$  provide a lot of new critical initial spaces where the well-posedness of equations (\ref{eqn:ns}) holds.
 By Lemma \ref{lem:2.9}, for $\beta=1$,  Theorem \ref{mthmain} holds for the initial spaces $\dot{B}^{\gamma_{1}, \gamma_{2}}_{p,q}$ satisfying
 $$\dot{B}^{-1+w, w}_{2,q'}\subset Q^{-1}_{\alpha}(\mathbb{R}^{n}),\ q'<2, w>0$$
 or
 $$ Q^{-1}_{\alpha}(\mathbb{R}^{n})\subset \dot{B}^{-1+w, w}_{2,q''}\subset BMO^{-1}(\mathbb{R}^{n}),\ 2<q''<\infty, w>0.$$
 In some sense, $\dot{B}^{\gamma_{1}, \gamma_{2}}_{p,q}$ fills the gap between the critical spaces $Q^{-1}_{\alpha}(\mathbb{R}^{n})$ and $BMO^{-1}(\mathbb{R}^{n})$. See also Lemma \ref{lem:2.9}, Corollary \ref{co:1} and Remark
\ref{remark1}.

\item{\rm(iii)}
For  a initial data $ a\in\dot{B}^{\gamma_{1}, \gamma_{2}}_{p,q}$, the index $\gamma_{1}$ represents the regularity of  $a$.
In Theorem \ref{mthmain}, taking $\dot{B}^{\gamma_{1}, \gamma_{2}}_{p,q}=\dot{B}^{-1+w, w}_{p,q}$ with $w>0$, $p>2$ and $q>2$ yields
$\dot{B}^{-1+w, w}_{p,q}\subset BMO^{-1}(\mathbb{R}^{n})$. Compared with the ones in $BMO^{-1}(\mathbb{R}^{n})$, the elements of $\dot{B}^{-1+w, w}_{p,q}$ have higher regularity. Furthermore,
Theorem \ref{mthmain} implies that the regularity of our solutions become higher along with the growth of $\gamma_{1}$.

\item{\rm(iv)} Interestingly, Federbush \cite{Feder} employed the divergence-free wavelets to study the classical Navier-Stokes equations, while the wavelets used in this paper are classical Meyer wavelets. In addition, when constructing a contraction mapping, Federbush's method was based on the estimates of ``long wavelength residues".  Nevertheless, our wavelet approach based on Lemmas \ref{le6}-\ref{le7} and the Cauchy-Schwarz inequality is to convert the bilinear estimate of $B(u,v)$ into various efficient computations involved in the wavelet coefficients of $u$ and $v$.
\end{remark}

The remaining of this paper is organized as follows. In Section
\ref{sec2}, we list some preliminary knowledge on wavelets and give
the wavelet characterization of the Besov-Q spaces. In Sections
\ref{sec3}-\ref{sec4} we define the initial data spaces and the
corresponding solution spaces. Section \ref{sec5} carries out
a necessary analysis of some non-linear terms and a prior estimates. In
Section \ref{sec6}, we verify Theorem \ref{mthmain} via Lemmas \ref{le6}-\ref{le7}
which will be demonstrated in Sections \ref{sec7}-\ref{sec8} respectively.

{\it Notation}:\  ${\mathsf U}\approx{\mathsf V}$ represents that
there is a constant $c>0$ such that $c^{-1}{\mathsf V}\le{\mathsf
U}\le c{\mathsf V}$ whose right inequality is also written as
${\mathsf U}\lesssim{\mathsf V}$. Similarly, one writes ${\mathsf V}\gtrsim{\mathsf U}$ for
${\mathsf V}\ge c{\mathsf U}$.

\section{Some preliminaries}\label{sec2}

First of all, we would like to say that we will always utilize tensorial product real-valued
orthogonal wavelets which may be regular Daubechies wavelets (only used for characterizing Besov and Besov-Q spaces) and
classical Meyer wavelets, but also to recall that the regular Daubechies wavelets are such Daubechines wavelets that are smooth enough and have more sufficient vanishing moments than the relative spaces do; see Lemma \ref{le9} and the part before Lemma \ref{lem:c}.

Next, we present some preliminaries on Meyer wavelets $\Phi^{\epsilon}(x)$ in detail
and refer the reader to \cite{Me}, \cite{Woj} and
\cite{Yang1} for further information. Let $\Psi^{0}$ be an even function in $ C^{\infty}_{0}
([-\frac{4\pi}{3}, \frac{4\pi}{3}])$ with
\begin{equation}
\left\{ \begin{aligned}
&0\leq\Psi^{0}(\xi)\leq 1; \nonumber\\
&\Psi^{0}(\xi)=1\text{ for }|\xi|\leq \frac{2\pi}{3}.\nonumber
\end{aligned} \right.
\end{equation}
If
  $$\begin{array}{rl}
  \Omega(\xi)= \sqrt{(\Psi^{0}(\frac{\xi}{2}))^{2}-(\Psi^{0}(\xi))^{2}},
  \end{array}
  $$
  then $\Omega$ is an even
function in $ C^{\infty}_{0}([-\frac{8\pi}{3}, \frac{8\pi}{3}])$.
Clearly,
\begin{equation}
\left\{ \begin{aligned}
&\Omega(\xi)=0\text{ for }|\xi|\leq \frac{2\pi}{3};\nonumber\\
&\Omega^{2}(\xi)+\Omega^{2}(2\xi)=1=\Omega^{2}(\xi)+\Omega^{2}(2\pi-\xi)\text{
for }\xi\in [\frac{2\pi}{3},\frac{4\pi}{3}].\nonumber
\end{aligned} \right.
\end{equation}

 Let $\Psi^{1}(\xi)=
\Omega(\xi) e^{-\frac{i\xi}{2}}$. For any $\epsilon=
(\epsilon_{1},\cdots, \epsilon_{n}) \in \{0,1\}^{n}$, define
$\Phi^{\epsilon}(x)$ via the Fourier transform $\hat{\Phi}^{\epsilon}(\xi)=
\prod\limits^{n}_{i=1} \Psi^{\epsilon_{i}}(\xi_{i})$. For $j\in
\mathbb{Z}$ and $k\in\mathbb{Z}^{n}$, set $\Phi^{\epsilon}_{j,k}(x)=
2^{\frac{nj}{2}} \Phi^{\epsilon} (2^{j}x-k)$. Furthermore, we put
\begin{equation}\nonumber
\left\{ \begin{aligned}
E_{n}&=\{0,1\}^{n}\backslash\{0\}; \\
F_{n}&=\{(\epsilon,k):\epsilon\in E_{n}, k\in\mathbb{Z}^{n}\};\\
\Lambda_{n}&=\{(\epsilon,j,k), \epsilon\in E_{n}, j\in\mathbb{Z},
k\in \mathbb{Z}^{n}\},
\end{aligned} \right.
\end{equation}
and for any $\epsilon\in \{0,1\}^{n}, k\in \mathbb{Z}^{n}$ and a function
$f$ on $\mathbb R^n$, we write $f^{\epsilon}_{j,k}= \langle f,
\Phi^{\epsilon}_{j,k}\rangle .$

The following result is well-known.
\begin{lemma}\label{le1}
The Meyer wavelets $\{\Phi^{\epsilon}_{j,k}\}_{(\epsilon,j,k)\in
\Lambda_{n}}$ form an  orthogonal basis in $L^{2}(\mathbb{R}^{n})$.
Consequently, for any $f\in L^{2}(\mathbb{R}^{n})$, the following wavelet
decomposition holds in the $L^2$ convergence sense:
$$\begin{array}{c}
f=\sum\limits_{(\epsilon,j,k)\in\Lambda_{n}}f^{\epsilon}_{j,k}\Phi^{\epsilon}_{j,k}.
\end{array}$$
\end{lemma}

Moreover, for $j\in\mathbb{Z}$, let
$$
\begin{array}{rl}
&P_{j}f= \sum\limits_{k\in \mathbb{Z}^{n}} f^{0}_{j,k}
\Phi^{0}_{j,k}\ \ \text{ and }\ \ Q_{j}f= \sum\limits_{(\epsilon,k)\in
F_{n}} f^{\epsilon}_{j,k} \Phi^{\epsilon}_{j,k}.
\end{array}
$$
For the above Meyer wavelets,  by Lemma \ref{le1}, the product of any two functions $u$
and $v$ can be decomposed as
\begin{equation}\label{eq:decompose}
\begin{array}{rl}
uv= & \sum\limits_{j\in \mathbb{Z}} P_{j-3}u Q_{j}v +
\sum\limits_{j\in \mathbb{Z}} Q_{j}u Q_{j}v
+ \sum\limits_{0<j-j'\leq 3} Q_{j}u Q_{j'}v \\
&+ \sum\limits_{0<j'-j\leq 3} Q_{j}u Q_{j'}v + \sum\limits_{j\in
\mathbb{Z}} Q_{j}u P_{j-3}v.
\end{array}
\end{equation}

Suppose that $\varphi$ is a function on $\mathbb{R}^{n}$ satisfying
\begin{equation}\nonumber
\left\{ \begin{aligned}
&\text{supp}\hat{\varphi}\subset\{\xi\in\mathbb{R}^{n}:
|\xi|\leq1\}; \\
&\hat{\varphi}(\xi)=1\text{ for }
\{\xi\in\mathbb{R}^{n}: |\xi|\leq\frac{1}{2}\},\\
\end{aligned} \right.
\end{equation}
and that
$$\varphi_{v}(x)=2^{n(v+1)}\varphi(2^{v+1}x)-2^{nv}\varphi(2^{v}x)\ \ \forall\ \ v\in \mathbb{Z},$$
are the Littlewood-Paley functions; see \cite{P}.

\begin{definition}
Given $-\infty<\alpha<\infty$, $0<p,q<\infty$. A function
$f\in \mathbb{S}'(\mathbb{R}^{n})/\mathcal{P}(\mathbb{R}^{n})$
belongs to $\dot{B}^{\alpha,q}_{p}(\mathbb{R}^{n})$ if
$$\begin{array}{rl}
\|f\|_{\dot{B}^{\alpha,q}_{p}}&=\Big[\sum\limits_{v\in\mathbb{Z}}2^{qv\alpha}\|\varphi_{v}\ast
f\|_{p}^{q}\Big]^{\frac{1}{q}}<\infty.
\end{array}$$
\end{definition}

The following lemma is essentially known.

\begin{lemma}\label{le9}{\rm (Meyer \cite{Me})} Let
$\{\Phi^{\epsilon,1}_{j,k}\}_{(\epsilon,j,k)\in\Lambda_{n}}$ and
$\{\Phi^{\epsilon,2}_{j,k}\}_{(\epsilon,j,k)\in\Lambda_{n}}$ be two
different wavelet bases which are sufficiently regular. If
$$a^{\epsilon,\epsilon'}_{j,k,j',k'}=\langle \Phi^{\epsilon,1}_{j,k},\ \Phi^{\epsilon',2}_{j',k'} \rangle,
$$
then for any natural number $N$ there
exists a positive constant $C_N$ such that for
$j,j'\in\mathbb{Z}$ and $k,k'\in\mathbb{Z}^{n}$,
\begin{equation}\label{eq8}
|a^{\epsilon,\epsilon'}_{j,k,j',k'}|\leq
C_{N}2^{-|j-j'|(\frac{n}{2}+N)}\Big(\frac{2^{-j}+2^{-j'}}{2^{-j}+2^{-j'}+|2^{-j}k-2^{-j'}k'|}\Big)^{n+N}.
\end{equation}

\end{lemma}

According to Lemma \ref{le9} and Peetre's paper \cite{P}, we see that this definition of
$\dot{B}^{\alpha,q}_{p}(\mathbb{R}^{n})$ is independent of the
choice of $\{\varphi_{v}\}_{v\in\mathbb{Z}}$, whence reaching the following description of
$\dot{B}^{\alpha,q}_{p}(\mathbb{R}^{n})$.
\begin{theorem}\label{th2}
Given $s\in\mathbb{R}$ and $0<p,q<\infty$. A function $f$
belongs to $\dot{B}^{s,q}_{p}(\mathbb{R}^{n})$ if and only if
$$\begin{array}{rl}
&\Big[\sum\limits_{j\in\mathbb{Z}}2^{qj(s+\frac{n}{2}-\frac{n}{p})}
\Big(\sum\limits_{\epsilon,k}|f^{\epsilon}_{j,k}|^{p}\Big)^{\frac{q}{p}}\Big]^{\frac{1}{q}}<\infty.
\end{array}$$
\end{theorem}

\section{Besov-Q spaces via wavelets}\label{sec3}

\subsection{Definition and its wavelet formulation} The forthcoming Besov-Q spaces cover many important function spaces, for example, Besov spaces, Morrey
spaces and Q-spaces and so on. Such spaces were first introduced by
wavelets in Yang \cite{Yang1} and were studied by several authors.
For a related overview, we refer to Yuan-Sickel-Yang \cite{YSY}.

Let $\varphi\in C^{\infty}_{0} (B(0,n))$ and $\varphi(x)=1 $ for
$x\in B(0, \sqrt{n})$. Let $Q(x_{0},r)$ be a cube parallel to the
coordinate axis, centered at $x_{0}$ and with side length $r$. For
simplicity, sometimes, we denote by $Q=Q(r)$ the cube $Q(x_{0},r)$
and let $\varphi_{Q}(x)= \varphi(\frac{x-x_{Q}}{r})$. For $1<p,q<
\infty$ and $\gamma_{1}, \gamma_{2}\in \mathbb{R}$, let
$m_{0}=m^{\gamma_{1},\gamma_{2}}_{p,q}$ be a positive constant large
enough. For arbitrary function $f$, let
$S^{\gamma_{1},\gamma_{2}}_{p,q,f}$  be the class of the polynomial
functions $P_{Q,f}$ such that
$$\begin{array}{rl}
\int x^{\alpha} \varphi_{Q}(x) (f(x)
-P_{Q,f}(x)) dx =0\ \ \forall\ \ |\alpha|\leq m_{0}.
\end{array}$$

\begin{definition}
Given $1<p,q<\infty$ and $\gamma_{1}, \gamma_{2}\in \mathbb{R}$. We say that $f$ belongs to
the Besov-Q space $\dot{B}^{\gamma_{1},\gamma_{2}}_{p,q}:=\dot{B}^{\gamma_{1},\gamma_{2}}_{p,q}(\mathbb{R}^{n})$ provided

\begin{equation}\label{eq:b}
\sup\limits_{Q}|Q|^{\frac{\gamma_{2}}{n}-\frac{1}{p}}
\inf\limits_{P_{Q, f}\in S^{\gamma_{1},\gamma_{2}}_{p,q,f}}
\|\varphi_{Q} (f-P_{Q,f})\|_{\dot{B}^{\gamma_{1}, q}_{p}} <\infty,
\end{equation}
where the superum is taken over all cubes $Q$ with center $x_{Q}$ and
length $r$.
\end{definition}

As a generalization of the Morrey spaces, the forthcoming Besov-Q spaces cover many important function spaces, for example, Besov spaces, Morrey spaces and Q-spaces and so on. Such spaces were first introduced by wavelets in Yang \cite{Yang1}. On the other hand, our Lemma \ref{lem:c} as below and Yang-Yuan's \cite[Theorem 3.1]{Ya-Yu} show that our Besov-Q spaces and their Besov type spaces coincide; see also Liang-Sawano-Ullrich-Yang-Yuan \cite{LSUYY} and Yuan-Sickel-Yang \cite{YSY} for more information on the so-called Yang-Yuan's spaces.

Given $1<p, q<\infty$ and $\gamma_{1}, \gamma_{2}\in \mathbb{R}$.
Let $m_{0}=m^{\gamma_{1},\gamma_{2}}_{p,q}$ be a sufficiently big integer.
For the regular Daubechies wavelets $\Phi^{\epsilon}(x)$,  there
exist two integers $m\geq m_{0}=m^{\gamma_{1},\gamma_{2}}_{p,q}$ and
$M$ such that for $\epsilon\in E_n$, $\Phi^{\epsilon}(x)\in
C^m_0([-2^M, 2^M]^n)$ and $\int x^{\alpha} \Phi^{\epsilon}(x)
dx=0\ \forall\ |\alpha|\leq m$.  By applying the regular Daubechies wavelets,
we have the following wavelet characterization for
$\dot{B}^{\gamma_{1},\gamma_{2}}_{p,q}$.

\begin{lemma}\label{lem:c}
\item{\rm (i)}
$f= \sum\limits_{\epsilon,j,k} a^{\epsilon}_{j,k}
\Phi^{\epsilon}_{j,k}\in
\dot{B}^{\gamma_{1},\gamma_{2}}_{p,q}$ if and only if
\begin{equation}\label{eq:c}
\begin{array}{rl}
&\sup\limits_{Q}|Q|^{\frac{\gamma_{2}}{n}-\frac{1}{p}}
\Big[\sum\limits_{nj\geq -\log_{2}|Q|}
2^{jq(\gamma_{1}+\frac{n}{2}-\frac{n}{p})}
\Big(\sum\limits_{(\epsilon,k):Q_{j,k}\subset Q}
|a^{\epsilon}_{j,k}|^{p}\Big) ^{ \frac{q}{p} }\Big]^{\frac{1}{q}} <+
\infty,\end{array}
\end{equation}
where the supremum is taken over all dyadic cubes in
$\mathbb{R}^{n}$.
\item{\rm (ii)} The wavelet characterization in (i) is also true for
the Meyer wavelets.
\end{lemma}

A direct application of Lemma \ref{lem:c} gives the following assertion.
\begin{corollary}
\label{cor:BMpropty} Given $1<p,q<\infty, \gamma_{1},
\gamma_{2}\in \mathbb{R}$.

\item{\rm(i)} Each $\dot{B}^{\gamma_{1},\gamma_{2}}_{p,q}$ is
a Banach space.

\item{\rm(ii)} The definition of
$\dot{B}^{\gamma_{1},\gamma_{2}}_{p,q}$ is
independent on the choice of $\phi$.

\end{corollary}

Now we recall some preliminaries on the Calder\'on-Zygmund
operators (cf. \cite{Me} and \cite{MY}). For $x\neq y$, let $K(x,y)$
be a smooth function such that there exists a sufficiently large
$N_{0}\leq m$ satisfying that
\begin{equation}\label{eq6}
\begin{array}{rl}
&|\partial ^{\alpha}_{x}\partial ^{\beta}_{y} K(x,y)| \lesssim
|x-y|^{-(n+|\alpha|+|\beta|)}\ \ \forall\ \ |\alpha|+ |\beta|\leq N_{0}.
\end{array}
\end{equation}

A linear operator
$$\begin{array}{rl}
&Tf(x)=\int K(x,y) f(y) dy\end{array}$$
is said to be a Calder\'on-Zygmund one if it is continuous from
$C^{1}(\mathbb{R}^{n})$ to $(C^{1}(\mathbb{R}^{n}))'$, where the
kernel $K(\cdot,\cdot)$ satisfies (\ref{eq6}) and
$$
Tx^{\alpha}=T^{*}x^{\alpha}=0\ \ \forall\ \ \alpha \in \mathbb{N}^{n}\ \ \hbox{with}\ \
|\alpha|\leq N_{0}.
$$
For such an operator, we denote $T\in
CZO(N_{0})$.

The kernel $K(\cdot,\cdot)$ may have high singularity on the diagonal $x=y$,
so according to the Schwartz kernel theorem, it is only a distribution in $S'(\mathbb{R}^{2n})$.
For $ (\epsilon,j,k), (\epsilon',j',k')\in \Lambda_{n}$, let
$$a^{\epsilon,\epsilon'}_{j,k,j',k'}= \langle K(x,y),
\Phi^{\epsilon}_{j,k}(x) \Phi^{\epsilon'}_{j',k'}(y)\rangle.$$
 If
$T$ is a Calder\'on-Zygmund operator, then its kernel $K(\cdot,\cdot)$ and
the related coefficients satisfy the following relations (cf.
\cite{Me}, \cite{MY} and \cite{Yang1}):

\begin{lemma}\label{le8}

\item{\rm(i)} If $T\in CZO(N_{0})$, then the coefficients
$a^{\epsilon,\epsilon'}_{j,k,j',k'}$ satisfy
\begin{equation}\label{eq7}
|a^{\epsilon,\epsilon'}_{j,k,j',k'}| \lesssim
\frac{
\Big(\frac{2^{-j}+2^{-j'}}{2^{-j}+2^{-j'}
+|k2^{-j}-k'2^{-j'}|}\Big)^{n+N_{0}}}{2^{|j-j'|(\frac{n}{2}+N_{0})}
}\ \forall\ (\epsilon,j,k), (\epsilon',j',k')\in \Lambda_{n}.
\end{equation}

\item{\rm(ii)} If $a^{\epsilon,\epsilon'}_{j,k,j',k'}$ satisfy (\ref{eq7}),
then $K(\cdot,\cdot)$, the kernel of the operator $T$, can be written as
$$\begin{array}{rl}
&K(x,y)=\sum\limits_{ (\epsilon,j,k),(\epsilon',j',k')\in
\Lambda_{n}} a^{\epsilon,\epsilon'}_{j,k,j',k'}\Phi^{\epsilon}_{j,k}(x)
\Phi^{\epsilon'}_{j',k'}(y)
\end{array}$$
in the distribution sense. Moreover, $T$ belongs to $CZO(N_{0}-\delta)$ for any small positive number $\delta$.
\end{lemma}
By Lemma \ref{le8}, we can prove that
\begin{corollary} \label{cor:dilation}
For any $\frac{1}{2}\leq \lambda \leq 2$ we have $\|f(\lambda
\cdot)\|_{\dot{B}^{\gamma_{1},\gamma_{2}}_{p,q}} \approx
\|f\|_{\dot{B}^{\gamma_{1},\gamma_{2}}_{p,q}}$.
\end{corollary}

\subsection{Critical spaces and their inclusions}
\begin{definition}
An initial data space is called critical for (\ref{eqn:ns}), if it
is invariant under the scaling $f_{\lambda}(x) = \lambda^{2\beta-1}
f(\lambda x)$.
\end{definition}

Note that, if $u(t,x)$ is a solution of (\ref{eqn:ns}) and we replace
$u(t,x), p(t,x), a(x)$ by
$$u_{\lambda}(t,x) = \lambda^{2\beta-1}
u(\lambda^{2\beta}t, \lambda x),\ p_{\lambda}(t,x) =
\lambda^{4\beta-2} u(\lambda^{2\beta}t, \lambda x)$$ and
$a_{\lambda}(x) = \lambda^{2\beta-1} a(\lambda x),$ respectively,
$u_{\lambda}(t,x)$ is also a solution of (\ref{eqn:ns}). So, the critical spaces
occupy a significant place for (\ref{eqn:ns}). For
$\beta=1$,
$$
\begin{cases}
\dot{L}^{2}_{\frac{n}{2}-1}(\mathbb{R}^{n})=
\dot{B}^{-1+\frac{n}{2},2}_{2}(\mathbb{R}^{n});\\
L^{n}(\mathbb{R}^{n});\\
\dot{B}^{-1+\frac{n}{p},
\infty}_{p}(\mathbb{R}^{n}), p<\infty;\\
BMO^{-1}(\mathbb{R}^{n});\\
\dot{B}^{\alpha-1,\alpha}_{2,2}(\mathbb{R}^{n}),
\end{cases}
$$
are critical spaces. For the general $\beta$,
$$
\begin{cases}
\dot{B}^{1+\frac{n}{p}-2\beta,
\infty}_{p}(\mathbb{R}^{n}), p<\infty;\\
\dot{B}^{\alpha-\beta+1,\alpha+\beta}_{2,2}(\mathbb{R}^{n}),
\end{cases}
$$
are critical spaces.

By Corollary \ref{cor:dilation}, it is easy to see that each $\dot{B}^{\gamma_{1},\gamma_{2}}_{p,q}$ enjoys following dilation-invariance.
For $\beta>\frac{1}{2}$ and $\gamma_{1}-\gamma_{2} = 1-2\beta$,
each $\dot{B}^{\gamma_{1},\gamma_{2}}_{p,q}$ is a critical
space, i.e.,
$$\|\lambda^{\gamma_{2}-\gamma_{1}} f(\lambda
\cdot)\|_{\dot{B}^{\gamma_{1},\gamma_{2}}_{p,q}} \approx
\|f\|_{\dot{B}^{\gamma_{1},\gamma_{2}}_{p,q}}\quad\forall\quad \lambda>0.$$

To better understand why the Besov-Q spaces are larger
than many spaces cited in the introduction, we should observe the basic fact below.

\begin{lemma}\label{lem:2.9}
Given $1<p, q<\infty$ and $\gamma_{1},\gamma_{2}\in\mathbb{R}$.

\item{\rm(i)} If $q_{1}\leq q_{2}$, then
$\dot{B}^{\gamma_{1},\gamma_{2}}_{p,q_{1}}\subset
\dot{B}^{\gamma_{1},\gamma_{2}}_{p,q_{2}}$.

\item{\rm(ii)}
$\dot{B}^{\gamma_{1},\gamma_{2}}_{p,q}\subset
\dot{B}^{\gamma_{1}-\gamma_{2},\infty}_{\infty}(\mathbb R^n)$.

\item{\rm(iii)} Given $p_1\geq 1$. For $w=0,q_1=1$ or $w>0,
1\leq q_1\leq \infty$, one has
$\dot{B}^{\gamma_{1},
\gamma_{2}+w}_{p,q} \subset
\dot{B}^{\gamma_{1}-w,\gamma_{2}} _{\frac{p}{p_1},
\frac{q}{q_1}}$.
\end{lemma}

For $0\leq \alpha-\beta+1$ and $\alpha+\beta-1\leq \frac{n}{2}$,
we say that $f$ belongs to the Q-type space $Q_{\alpha}^{\beta}(\mathbb{R}^{n})$ provided
$$
\begin{array}{rl}
\sup\limits_{Q} r^{2(\alpha+\beta-1)-n} \int_{Q}\int_{Q}
\frac{|f(x)-f(y)|^{2}}{|x-y|^{n+2(\alpha-\beta+1)}} dxdy<\infty,
\end{array}$$
where the supremum is taken over all cubes with sidelength $r$. This definition was used in \cite{LZ} to extend the results in \cite{X} which initiated a PDE-analysis of the original Q-spaces introduced in \cite{EJPX} (cf. \cite{DX, DX1}, \cite{PY}, \cite{WX}, and \cite{Yang1} for more information). The following is a direct consequence of Lemmas \ref{lem:c} and
\ref{lem:2.9}.

\begin{corollary}\label{co:1}
\item{\rm(i)} If $0\leq \alpha-\beta+1< 1, \alpha+\beta-1\leq \frac{n}{2}$,
then $Q_{\alpha}^{\beta}(\mathbb{R}^{n})= \dot{B}^{\alpha-\beta+1,
\alpha+\beta-1}_{2,2}$.

\item{\rm(ii)} If $p=\frac{n}{\gamma_{2}}$, then $\dot{B}^{\gamma_{1},
\gamma_{2}}_{p,q}=\dot{B}^{\gamma_{1},q}_{p}(\mathbb{R}^{n})$.

\item{\rm(iii)} Given $w=0,v=1$ or $w>0, 1\leq v\leq \infty$. If
$p=\frac{n}{\gamma_{2}+w}$, then
$$\begin{array}{l}
\dot{B}^{\gamma_{1}, q}_{p}(\mathbb{R}^{n}) \subset
\dot{B}^{\gamma_{1}-w,\gamma_{2}} _{\frac{n}{u(w+\gamma_{2})},
\frac{q}{v}}.
\end{array}$$
\end{corollary}

\begin{remark}\label{remark1}
In \cite{W2}, J. Wu got the well-posedness of
(\ref{eqn:ns}) with an initial data in the critical Besov space
$\dot{B}^{1+\frac{n}{p}-2\beta,q}_{p}(\mathbb{R}^{n})$. Given
$1<p_{0},q_{0}<\infty$.  By Lemma \ref{lem:2.9}, we can see that if
$1<p\leq p_{0}$, $1<q\leq q_{0}$ and $\beta>0$,
$$\dot{B}^{1+\frac{n}{p}-2\beta,q}_{p}(\mathbb{R}^{n})\subset
\dot{B}^{1+\frac{n}{p}-2\beta,\frac{n}{p}}_{p_{0},q_{0}}.$$
\end{remark}

\section{Besov-Q spaces via semigroups}\label{sec4}

To establish a semigroup characterization of the
Besov-Q spaces, recall the following semigroup characterization of $Q^{\beta}_{\alpha}(\mathbb{R}^{n})$, see \cite{LZ}:

{\it Given $\max\{\alpha,1/2\}<\beta<1$ and
$\alpha+\beta-1\geq0$. $f\in Q^{\beta}_{\alpha}(\mathbb{R}^{n})$
if and only if
$$\begin{array}{rl}
\sup\limits_{x\in\mathbb R^n\ \&\ r\in (0,\infty)}
r^{2\alpha-n+2\beta-2}\int^{r^{2\beta}}_{0} \int_{|y-x|<r} |\nabla
e^{-t(-\Delta)^{\beta}} f(y)|^{2} t^{-\frac{\alpha}{\beta}}
dydt<\infty.
\end{array}$$}

This characterization was used to derive the global existence and uniqueness of a mild solution to (\ref{eqn:ns}) with a small initial data in
$\nabla\cdot(Q^{\beta}_{\alpha}(\mathbb{R}^{n}))^{n}$. Notice that
$$
Q^{\beta}_{\alpha}(\mathbb{R}^{n})=\dot{B}^{\alpha+\beta-1,\alpha-\beta+1}_{2,2}.
$$
So, in order to get the corresponding result of (\ref{eqn:ns}) with a
small initial data in
$\dot{B}^{\gamma_{1},\gamma_{2}}_{p,q}$ under $|p-2|+
|q-2|\neq 0$, we need a more meticulous relation among time, frequency
and locality. For this purpose, by the Meyer wavelets and the
fractional heat semigroups, we introduce some new tent spaces
associated with
$\dot{B}^{\gamma_{1},\gamma_{2}}_{p,q}$, and then
establish some connections between these tent spaces and
$\dot{B}^{\gamma_{1},\gamma_{2}}_{p,q}$.

\subsection{Wavelets and semigroups}

For $\beta>0$, let $\hat{K}^{\beta}_{t} (\xi) =
e^{-t|\xi|^{2\beta}}$.
 We have
\begin{equation*}
f(t,x)= e^{-t(-\Delta)^{\beta}} f(x) = K_{t}^{\beta}*f(x).
\end{equation*}

 For the Meyer wavelets
$\{\Phi^{\epsilon}_{j,k}\}_{(\epsilon,j,k)\in\Lambda_{n}}$, let
$a^{\epsilon}_{j,k}(t) = \langle f(t,\cdot),
\Phi^{\epsilon}_{j,k}\rangle$
 and $a^{\epsilon}_{j,k} =
\langle f, \Phi^{\epsilon}_{j,k}\rangle$. By Lemma \ref{le1} we get
$$
\begin{array}{rl}
&f(x)= \sum\limits_{\epsilon,j,k} a^{\epsilon}_{j,k}
\Phi^{\epsilon}_{j,k}(x)\ \text{ and }\  f(t,x)=
\sum\limits_{\epsilon, j, k} a^{\epsilon}_{j,k}(t)
\Phi^{\epsilon}_{j,k}(x).
\end{array}
$$

If $f(t,x)=K^{\beta}_{t}*f(x)$, then

\begin{equation}\label{eq3}
\begin{array}{rl}
a^{\epsilon}_{j,k}(t) &= \sum\limits_{\epsilon',|j-j'|\leq 1, k'}
a^{\epsilon'}_{j',k'} \langle K^{\beta}_{t}
\Phi^{\epsilon'}_{j',k'},
\Phi^{\epsilon}_{j,k}\rangle\nonumber\\
&= \sum\limits_{\epsilon',|j-j'|\leq 1, k'} a^{\epsilon'}_{j',k'}
\int e^{-t 2^{2j\beta} |\xi|^{2\beta}} \hat{\Phi}^{\epsilon'}
(2^{j-j'}\xi) \hat{\Phi}^{\epsilon} (\xi) e^{-i(2^{j-j'}k'-k)\xi}
d\xi.\nonumber
\end{array}
\end{equation}

\begin{lemma}\label{le4} Let $\{\Phi^{\epsilon}_{j,k}\}_{(\epsilon,j,k)\in\Lambda_{n}}$ be Meyer wavelets.
For $\beta>0$ there exist a large constant $N_{\beta}>0$ and a small constant $\tilde c>0$ such that if $N>N_{\beta}$ then
\begin{equation}\label{eq3.5}
\begin{array}{l}
|a^{\epsilon}_{j,k}(t)|\lesssim e^{-\tilde c t 2^{2j\beta}}
\sum\limits_{\epsilon',|j-j'|\leq 1, k'} |a^{\epsilon'}_{j',k'}|
(1+|2^{j-j'}k'-k|)^{-N}\quad\forall\quad t 2^{2\beta j} \geq 1
\end{array} \end{equation}
and
\begin{equation}\label{eq3.6}
\begin{array}{rl}|a^{\epsilon}_{j,k}(t)|\lesssim  \sum\limits_{|j-j'|\leq 1}
\sum\limits_{\epsilon',k'} |a^{\epsilon'}_{j',k'}|
(1+|2^{j-j'}k'-k|)^{-N}\quad\forall\quad 0< t 2^{2\beta j} \leq 1.
\end{array}
\end{equation}
\end{lemma}
\begin{proof}
Formally, we can write
$$\begin{array}{rl}
a^{\epsilon}_{j,k}(t) =&\sum\limits_{\epsilon',|j-j'|\leq 1,
k'}a^{\epsilon'}_{j',k'}\Big\langle
K^{\beta}_{t}\ast\Phi^{\epsilon'}_{j',k'},
\Phi^{\epsilon}_{j,k}\Big\rangle\\
=&\sum\limits_{\epsilon',|j-j'|\leq 1,
k'}a^{\epsilon'}_{j',k'}\Big\langle
e^{-t(-\Delta)^{\beta}}\Phi^{\epsilon'}_{j',k'},
\Phi^{\epsilon}_{j,k}\Big\rangle\\
=&\sum\limits_{\epsilon',|j-j'|\leq 1,
k'}a^{\epsilon'}_{j',k'}\int e^{-t2^{2j\beta}|\xi|^{2\beta}}\widehat{\Phi^{\epsilon'}}(2^{j-j'}\xi)\widehat{\Phi^{\epsilon}}(\xi)e^{-i(2^{j-j'}k'-k)\xi}d\xi.
\end{array}$$
We divide the rest of the argument into two cases.

{\it Case 1:} $|2^{j-j'}k'-k|\le 2$. Notice that $\widehat{\Phi^{\epsilon}}$ is supported on a ring. By a direct computation, we can get
$$\begin{array}{rl}
|a^{\epsilon}_{j,k}(t)|\lesssim&\sum\limits_{\epsilon',|j-j'|\leq 1,
k'}|a^{\epsilon'}_{j',k'}|e^{-t2^{2j\beta}}\\
\lesssim&\sum\limits_{\epsilon',|j-j'|\leq 1,
k'}\frac{|a^{\epsilon'}_{j',k'}|}{(1+|2^{j-j'}k'-k|)^{N}}e^{-t2^{2j\beta}}.
\end{array}$$

{\it Case 2:} $|2^{j-j'}k'-k|\geq 2$. Denote by $l_{i_{0}}$ the largest component of $2^{j-j'}k'-k$. Then $(1+|l_{i_{0}}|)^{N}\sim (1+|2^{j-j'}k'-k|)^{N}$. We have
$$\begin{array}{rl}
a^{\epsilon}_{j,k}(t)=&\sum\limits_{\epsilon',|j-j'|\leq 1,
k'}\frac{a^{\epsilon'}_{j',k'}}{(l_{i_{0}})^{N}}\int e^{-t2^{2j\beta}|\xi|^{2\beta}}\widehat{\Phi^{\epsilon'}}(2^{j-j'}\xi)\widehat{\Phi^{\epsilon}}(\xi)\\
&\times[(\frac{-1}{i}\partial_{\xi_{i_{0}}})^{N}e^{-i(2^{j-j'}k'-k)\xi}]d\xi.
\end{array}$$
By an integration-by-parts, we can obtain that if $C^l_N$ is the binomial coefficient indexed by $N$ and $l$ then
$$\begin{array}{rl}
|a^{\epsilon}_{j,k}(t)|=&\Big|\sum\limits_{\epsilon',|j-j'|\leq 1,
k'}(-1)^{N}\frac{a^{\epsilon'}_{j',k'}}{(l_{i_{0}})^{N}}\int \sum\limits^{N}_{l=0}C^{l}_{N}\partial_{\xi_{i_{0}}}^{l}(e^{-t2^{2j\beta}|\xi|^{2\beta}})\\
&\times\partial_{\xi_{i_{0}}}^{N-l}(\widehat{\Phi^{\epsilon'}}(2^{j-j'}\xi)\widehat{\Phi^{\epsilon}}(\xi))e^{-i(2^{j-j'}k'-k)\xi}d\xi\Big|\\
\lesssim&\sum\limits_{\epsilon',|j-j'|\leq 1,
k'}\frac{|a^{\epsilon'}_{j',k'}|}{(1+|2^{j-j'}k'-k|)^{N}}\Big|\int \sum\limits^{N}_{l=0}C^{l}_{N}(-t2^{2j\beta}|\xi|^{2\beta})^{l}|\xi|^{2\beta-2}\xi_{i_{0}}e^{-t2^{2j\beta}|\xi|^{2\beta}}\\
&\times\partial_{\xi_{i_{0}}}^{N-l}(\widehat{\Phi^{\epsilon'}}(2^{j-j'}\xi)\widehat{\Phi^{\epsilon}}(\xi))e^{-i(2^{j-j'}k'-k)\xi}d\xi\Big|.
\end{array}$$
If $t2^{2j\beta}\ge1$, there exists a constant $c$ such that
$(t2^{2j\beta})^{l}e^{-t2^{2j\beta}}\lesssim e^{-ct2^{2j\beta}}.$
Since $\widehat{\Phi^{\epsilon'}}$ is defined on a ring, we can get
$$\begin{array}{rl}
|a^{\epsilon}_{j,k}(t)|\lesssim&\sum\limits_{\epsilon',|j-j'|\leq 1,
k'}\frac{|a^{\epsilon'}_{j',k'}|}{(1+|2^{j-j'}k'-k|)^{N}}(t2^{2j\beta})^{l}e^{-t2^{2j\beta}}\\
\lesssim&\sum\limits_{\epsilon',|j-j'|\leq 1,
k'}e^{-ct2^{2j\beta}}\frac{|a^{\epsilon'}_{j',k'}|}{(1+|2^{j-j'}k'-k|)^{N}}.
\end{array}$$
If $0<t2^{2j\beta}\le1$, then we can directly deduce
$$\begin{array}{rl}
|a^{\epsilon}_{j,k}(t)|\lesssim\sum\limits_{\epsilon',|j-j'|\leq 1,
k'}\frac{|a^{\epsilon'}_{j',k'}|}{(1+|2^{j-j'}k'-k|)^{N}}.
\end{array}$$
\end{proof}

\subsection{Tent spaces generated by Besov-Q spaces} Over $\mathbb{R}^{1+n}_{+}$ we introduce a new
tent type space $\B^{\gamma_{1},\gamma_{2}}_{p,q,m,m'}$ associated
with $\dot{B}^{\gamma_{1},\gamma_{2}}_{p,q}$, and then establish a relation between
$\dot{B}^{\gamma_{1},\gamma_{2}}_{p,q}$ and
$\B^{\gamma_{1},\gamma_{2}}_{p,q,m,m'}$ via the fractional heat
semigroup $e^{-t(-\Delta)^{\beta}}$.

\begin{definition}
Let $\gamma_{1}, \gamma_{2},m\in \mathbb{R}$, $m' >0$, $1<p,
q<\infty$ and
$$
a(t,x)= \sum\limits_{(\epsilon,j,k)\in \Lambda_{n}}
 a^{\epsilon}_{j,k}(t) \Phi^{\epsilon}_{j,k}(x).
 $$ We say that:
\item{(i)} $f\in \B^{\gamma_{1},\gamma_{2},I}_{p,q ,m }$ if $\sup\limits_{t\geq 0} \sup\limits_{ x_{0},r}I^{\gamma_{1},\gamma_{2}}_{p, q,Q_{r},m}(t) <\infty,$ where
$$\begin{array}{rl}
I^{\gamma_{1},\gamma_{2}}_{p, q,Q_{r},m}(t)=:|Q_{r}
|^{\frac{q\gamma_{2}}{n}-\frac{q}{p}} \sum\limits_{j\geq \max\{
-\log_{2}r, \frac{-\log_{2}t}{2\beta}\} }
2^{jq(\gamma_{1}+\frac{n}{2}-\frac{n}{p})}\Big[\sum\limits_{(\epsilon,k):Q_{j,k}\subset Q_{r}}
|a^{\epsilon}_{j,k}(t)|^{p}(t2^{2j\beta})^{m}
\Big]^{\frac{q}{p}};
\end{array}$$

\item{(ii)} $f\in \B^{\gamma_{1},\gamma_{2},II}_{p, q}$ if
$\sup\limits_{t\geq 0} \sup\limits_{x_{0},r}
II^{\gamma_{1},\gamma_{2}}_{p, q, Q_{r}}(t)<\infty$, where
$$\begin{array}{rl}
II^{\gamma_{1},\gamma_{2}}_{p, q, Q_{r}}(t)=:|Q_{r}
|^{\frac{q\gamma_{2}}{n}-\frac{q}{p}} \sum\limits_{-\log_{2}r\leq
j<\frac{-\log_{2}t}{2\beta} }
2^{jq(\gamma_{1}+\frac{n}{2}-\frac{n}{p})}\Big[\sum\limits_{(\epsilon,k):Q_{j,k}\subset Q_{r}}
|a^{\epsilon}_{j,k}(t)|^{p}\Big]^{\frac{q}{p}};
\end{array}$$

\item{(iii)}  $f\in \B^{\gamma_{1},\gamma_{2},III}_{p, q, m }$ if
$\sup\limits_{x_{0},r} III^{\gamma_{1},\gamma_{2}}_{p, q,
Q_{r},m}<\infty$, where
$$\begin{array}{rl}
III^{\gamma_{1},\gamma_{2}}_{p, q, Q_{r},m} =|Q_{r}
|^{\frac{q\gamma_{2}}{n}-\frac{q}{p}} \sum\limits_{j\geq -\log_{2}r
} 2^{jq(\gamma_{1}+\frac{n}{2}-\frac{n}{p})}
\left(\int^{r^{2\beta}}_{2^{-2j\beta}} \sum\limits_{(\epsilon,k):Q_{j,k}\subset Q_{r}}
|a^{\epsilon}_{j,k}(t)|^{p}(t2^{2j\beta})^{m}
\frac{dt}{t}\right)^{\frac{q}{p}};
\end{array}$$

\item{(iv)} $f\in \B^{\gamma_{1},\gamma_{2},IV}_{p,
q, m' }$ if $\sup\limits_{x_{0},r} IV^{\gamma_{1},\gamma_{2}}_{p, q,
Q_{r},m'} <\infty$, where
$$\begin{array}{rl}
 IV^{\gamma_{1},\gamma_{2}}_{p, q,
Q_{r},m'}=:|Q_{r}
|^{\frac{q\gamma_{2}}{n}-\frac{q}{p}} \sum\limits_{j\geq -\log_{2}r
} 2^{jq(\gamma_{1}+\frac{n}{2}-\frac{n}{p})}
\left(\int^{2^{-2j\beta}}_{0} \sum\limits_{(\epsilon,k):Q_{j,k}\subset Q_{r}}
|a^{\epsilon}_{j,k}(t)|^{p}(t2^{2j\beta})^{m'}
\frac{dt}{t}\right)^{\frac{q}{p}}.
\end{array}$$
Moreover, the associated tent type spaces are defined as

$$\begin{array}{rl}\B^{\gamma_{1},\gamma_{2}}_{p,q,
m,m' }= \B^{\gamma_{1},\gamma_{2},I}_{p, q, m }\bigcap
\B^{\gamma_{1},\gamma_{2},II}_{p,q  }\bigcap
\B^{\gamma_{1},\gamma_{2},III}_{p, q, m }\bigcap
\B^{\gamma_{1},\gamma_{2},IV}_{p, q, m' }.
\end{array}$$
\end{definition}

To continue our discussion, we need to introduce two more function spaces
$\B^{\gamma}_{\tau,\infty}$ and $\B^{\gamma_{1}}_{\tau,\infty}$.

\begin{definition}\label{Besov-infinity}
  For $(\epsilon,j,k)\in \Lambda_{n}$, write
$a^{\epsilon}_{j,k}(t)=\langle a(t,\cdot),
\Phi^{\epsilon}_{j,k}(\cdot)\rangle$. Given $\tau>0$ and $\gamma\in\mathbb{R}$. We say that
$$
\begin{cases}
a(\cdot,\cdot)\in \B^{\gamma}_{\tau,\infty}\ \ \hbox{if}\ \
\sup\limits_{t2^{2j\beta}\geq 1}
(t2^{2j\beta})^{\tau} 2^{\frac{nj}{2}}
2^{j\gamma}|a^{\epsilon}_{j,k}(t)| +\sup\limits_{0<t2^{2j\beta}<1}
2^{\frac{nj}{2}} 2^{j\gamma}|a^{\epsilon}_{j,k}(t)|<\infty;\\
a(\cdot,\cdot)\in\B^{\gamma}_{0,\infty}\ \ \hbox{if}\ \
t^{\frac{-\gamma}{2\beta}} 2^{\frac{nj}{2}}|\langle a(t,\cdot),
\Phi^{0}_{j,k}\rangle|<\infty.
\end{cases}
$$
\end{definition}

It is easy to verify the following inclusions.

\begin{lemma}\label{le:tau}
Given $1<p, q<\infty$, $\gamma_{1},\gamma_{2}\in \mathbb{R}$,
$m>p$ and $m',\tau>0$.

\item{\rm(i)} If $m>0$, then $\B^{\gamma_{1},\gamma_{2}}_{p, q, m,m'}
\subset\B^{\gamma_{1}-\gamma_{2}}_{\frac{m}{p},\infty}$.
\item{\rm(ii)} If $-2\beta\tau<\gamma<0<\beta$, then $\B^{\gamma}_{\tau,\infty}\subset
\B^{\gamma}_{0,\infty}.$
\end{lemma}

In sake of our convenience, for any dyadic cube $Q_{j_{0}, k_{0}}$, we always use $\widetilde{Q}_{j_{0},k_{0}}$ to denote the dyadic cube containing $Q_{j_{0}, k_{0}}$ with side length $2^{8-j_{0}}$.
Given
$(\epsilon,j,k)\in \Lambda_{n}$. If $\epsilon\in E_{n}$ and
$Q_{j,k}\subset Q_{j_{0}, k_{0}}$, we write $(\epsilon,k)\in
S^{j}_{j_{0},k_{0}}$.
For any $w\in \mathbb{Z}^{n}$, denote
$\widetilde{Q}^w_{j_{0},k_{0}}= 2^{8-j_{0}}w + \widetilde{Q}_{j_{0},
k_{0}}$. Denote $(\epsilon,k)\in S^{w,j}_{j_{0},k_{0}}$ whenever
$Q_{j,k}\subset \widetilde{Q}^w_{j_{0},k_{0}}$. Furthermore, we frequently utilize the so-called $\alpha$-triangle inequality below:
$$
(a+b)^{\alpha}\leq a^{\alpha}+b^{\alpha}\quad\forall\quad (\alpha,a,b)\in (0,1]\times(0,\infty)\times(0,\infty).
$$

Now we characterize the Besov-Q spaces by using a
semigroup operator.
\begin{theorem}\label{th1}
Given  $1<p<m<\infty$, $m'>0$, $1<q<\infty$,
$\gamma_{1}-\gamma_{2}<0 <\beta$. If $f\in \dot{B}^{\gamma_{1},\gamma_{2}}_{p,q
}$, then $f*K^{\beta}_{t}\in
\B^{\gamma_{1},\gamma_{2}}_{p, q, m,m'}$.
\end{theorem}

\begin{proof} We will prove
$$\begin{array}{rl}
&f= \sum\limits_{(\epsilon,j,k)\in
\Lambda_{n}} a^{\epsilon}_{j,k} \Phi^{\epsilon}_{j,k}\in
\dot{B}^{\gamma_{1},\gamma_{2}}_{p,q}\Longrightarrow f*K^{\beta}_{t}=
\sum\limits_{(\epsilon,j,k)\in \Lambda_{n}} a^{\epsilon}_{j,k}(t)
\Phi^{\epsilon}_{j,k} \in \B^{\gamma_{1},\gamma_{2}}_{p,q, m,m'}.\end{array}
$$
via handling four situations.

{\bf Situation 1}: $K^{\beta}_{t}\ast f\in \B^{\gamma_{1},\gamma_{2}, I}_{p,q,m}$.
For $t2^{2j\beta}>1, m>0$, by (\ref{eq3.5}), there
exists a constant $N$ large enough such that
$$\begin{array}{rl}
|a^{\epsilon}_{j,k}(t)|\lesssim&e^{-\tilde c t 2^{2j\beta}}
\sum\limits_{\epsilon',|j-j'|\leq 1, k'}\frac{
|a^{\epsilon'}_{j',k'}|}{ (1+|2^{j-j'}k'-k)|)^{N}} \lesssim
2^{\frac{-nj}{2}} 2^{j(\gamma_{2}-\gamma_{1})}e^{-\tilde c t
2^{2j\beta}}.
\end{array}$$
Choosing a sufficiently large $N'$ (depending on $N$) in the last estimate we have
$$\begin{array}{rl}
I^{\gamma_{1},\gamma_{2}}_{p,q, Q_{r},
m}(t)
&\lesssim \quad
|Q_{r}|^{\frac{q\gamma_{2}}{n}-\frac{q}{p}}\sum\limits_{j\geq\max\{-\log_{2}r,-\frac{\log_{2}t}{2\beta}\}}
2^{qj(\gamma_{1}+\frac{n}{2}-\frac{n}{p})}\\
& \quad\Big[\sum\limits_{(\epsilon,k)\in
S^{j}_{r}}e^{-\tilde{c}t2^{2j\beta}}(\sum\limits_{\epsilon',|j-j'|\leq1,
k'}\frac{|a^{\epsilon'}_{j',k'}|}{(1+|2^{j-j'}k'-k|)^{N}})^{p}(t2^{2j\beta})^{m}\Big]^{\frac{q}{p}},
\end{array}$$
where $p>1$ has been used. In the sequel, we divide the proof into two cases.

Case 1.1: $q\le p$. Because $|j-j'|\leq1$ and $j>-\log_{2}r$, one gets
$2^{-(j'+1)n}\<r^{n}$. This implies that $(2^{nj'}|Q|)^{-N'}\lesssim1$. Hence
$$\begin{array}{rl}
I^{\gamma_{1},\gamma_{2}}_{p,q,Q_{r}, m}(t)
 &\lesssim \!\!\!
\sum\limits_{w\in\mathbb{Z}^{n}}\frac{|Q_{r}|^{\frac{q\gamma_{2}}{n}-\frac{q}{p}}}{(1+|w|)^{{Nq}/{p}}}\sum\limits_{j'\geq\max\{-\log_{2}r,-\frac{\log_{2}t}{2\beta}\}}
2^{qj'(\gamma_{1}+\frac{n}{2}-\frac{n}{p})}
[\sum\limits_{(\epsilon',k')\in
S^{w,j'}_{r}}|a^{\epsilon'}_{j',k'}|^{p}]^{\frac{q}{p}}\\
&\lesssim\quad \|f\|_{\dot{B}^{\gamma_{1},\gamma_{2}}_{p,q}}.
\end{array}
$$

Case 1.2: $q>p$. Applying H\"{o}lder's inequality to
$w\in\mathbb{Z}^{n}$, we similarly have
$$\begin{array}{rl}
I^{\gamma_{1},\gamma_{2}}_{p,q,Q_{r},
m}(t)
&\lesssim\ \sum\limits_{w\in\mathbb{Z}^{n}}\Big\{|Q_{r}|^{\frac{q\gamma_{2}}{p}-\frac{q}{p}}
\sum\limits_{j'\geq\max\{-\log_{2}r,-\frac{\log_{2}t}{2\beta}\}}2^{qj'(\gamma_{1}+\frac{n}{2}-\frac{n}{p})}
[\sum\limits_{(\epsilon',k')\in
S^{w,j'}_{r}}|a^{\epsilon'}_{j',k'}|^{p}]^{\frac{q}{p}}\Big\}\\
&\lesssim\ \|f\|_{\dot{B}^{\gamma_{1},\gamma_{2}}_{p,q}}.
\end{array}$$

{\bf Situation 2}: $K^{\beta}_{t}\ast f\in \B^{\gamma_{1},\gamma_{2}, II}_{p,q}$.
For $t2^{2\beta j}\leq1$ and $m'>0$, by (\ref{eq3.6}), there exists a natural number $N$ large enough
such that $N>2n$ and
$$\begin{array}{rl}
|a^{\epsilon}_{j,k}(t)|\lesssim \sum\limits_{\epsilon', |j-j'|\leq1,
k'}|a^{\epsilon'}_{j',k'}|(1+|2^{j-j'}k'-k|)^{-N}\end{array}.$$
Consequently, we have
$$\begin{array}{rl}
II^{\gamma_{1},\gamma_{2}}_{p,q,Q_{r}}(t)\lesssim&|Q|^{\frac{q\gamma_{2}}{n}-\frac{q}{p}}\sum\limits_{-\log_{2}r\leq
j\leq-\frac{\log_{2}t}{2\beta}}2^{jq(\gamma_{1}+\frac{n}{2}-\frac{n}{p})}\Big[\sum\limits_{(\epsilon,k)\in
S^{j}_{r}}\Big(\sum\limits_{\epsilon',|j'-j|\leq1,
k'}\frac{|a^{\epsilon'}_{j',k'}|}{(1+|2^{j-j'}k'-k|)^{N}}\Big)^{p}\Big]^{\frac{q}{p}}.
\end{array}$$
Case 2.1: $q\le p$. Notice that $(2^{nj'}|Q|)^{-N}\lesssim1$. We have
$$\begin{array}{rl}
II^{\gamma_{1},\gamma_{2}}_{p,q,Q_{r}}(t)\lesssim&\!\!\!\sum\limits_{|w|\in\mathbb{Z}^{n}}
\frac{|Q_{r}|^{\frac{q\gamma_{2}}{n}-\frac{q}{p}}}{(1+|w|)^{qN/p}}
\sum\limits_{-\log_{2}r-1\leq
j'\leq-\frac{\log_{2}t}{2\beta}-1}2^{j'q(\gamma_{1}+\frac{n}{2}-\frac{n}{p})}
\Big(\sum\limits_{(\epsilon', k')\in
S^{w,j'}_{r}}|a^{\epsilon'}_{j',k'}|^{p}\Big)^{\frac{q}{p}}\\
\lesssim&\!\!\!\|f\|_{\dot{B}^{\gamma_{1},\gamma_{2}}_{p,q}}.
\end{array}$$
Case 2.2: $q>p$. By H\"{o}lder's inequality, by a similar manner, we can obtain
$$\begin{array}{rl}
II^{\gamma_{1},\gamma_{2}}_{p,q,Q_{r}}(t)\lesssim&\!\!\!
|Q_{r}|^{\frac{q\gamma_{2}}{n}-\frac{q}{p}}\sum\limits_{-\log_{2}r-1\leq
j'\leq-\frac{\log_{2}t}{2\beta}-1}2^{j'q(\gamma_{1}+\frac{n}{2}-\frac{n}{p})}\\
& \Big[\sum\limits_{(\epsilon',k')\in
S^{w,j'}_{Q_{r}}}|a^{\epsilon'}_{j',k'}|^{p}(1+|2^{j-j'}k'-k|)^{-N'}\Big]^{\frac{q}{p}}\\
\lesssim&\!\!\!\|f\|_{\dot{B}^{\gamma_{1},\gamma_{2}}_{p,q}}.
\end{array}$$

{\bf Situation 3}: $K^{\beta}_{t}\ast f\in \B^{\gamma_{1},\gamma_{2},
III}_{p,q,m}$. For this case we have $2^{-2j\beta}<t<r^{2\beta}$ and thus
$$\begin{array}{rl}
|a^{\epsilon}_{j,k}(t)|\lesssim&e^{-ct2^{2j\beta}}\sum\limits_{\epsilon',
|j-j'|\leq1,k'}\frac{|a^{\epsilon'}_{j',k'}|}{(1+|2^{j-j'}k'-k|)^{N'}}
\lesssim2^{-\frac{nj}{2}}2^{j(\gamma_{2}-\gamma_{1})}(t2^{2j\beta})^{-\tau}.
\end{array}$$
This yields
$$\begin{array}{rl}
III^{\gamma_{1},\gamma_{2}}_{p,q,Q_{r},m}
&\lesssim\ |Q_{r}|^{\frac{q\gamma_{2}}{n}-\frac{q}{p}}\sum\limits_{j\geq-\log_{2}r}
2^{jq(\gamma_{1}+\frac{n}{2}-\frac{n}{p})}\Big[\int^{r^{2\beta}}_{2^{-2j\beta}}(t2^{2j\beta})^{m}\\
&\ \sum\limits_{(\epsilon,k)\in S^{j}_{r}}e^{-cpt2^{2j\beta}}\Big(\sum\limits_{\epsilon',|j-j'|\leq 3,
k'}|a^{\epsilon'}_{j',k'}|(1+|2^{j-j'}k'-k|)^{-N}\Big)^{p}\frac{dt}{t}\Big]^{\frac{q}{p}}\\
\end{array}$$
Notice that $j\sim j'$ and the number of $\epsilon'$ is finite. Applying H\"older's inequality on $k'$ we obtain
$$\begin{array}{rl}
\Big[\sum\limits_{\epsilon',|j-j'|\leq3, k'}\frac{|a^{\epsilon'}_{j',k'}|}{(1+|2^{j-j'}k'-k|)^{N}}\Big]^{p}
&\lesssim \sum\limits_{|j-j'|\leq3}\sum\limits_{\epsilon', k'}\frac{|a^{\epsilon'}_{j',k'}|^{p}}{(1+|2^{j-j'}k'-k|)^{N}}.
\end{array}$$
Let $Q_{j,k}$ and $Q_{j',k'}$ be two dyadic cubes. Denote by $\widetilde{Q}_{j,k}$ the dyadic cube containing $Q_{j,k}$ with side length $2^{8-j}$. For $w\in\mathbb{Z}^{n}$, denote by $Q^{w}_{j,k}$ the cube $\widetilde{Q}_{j,k}+2^{8-j}w$.
It is easy to see that if $Q_{j',k'}\subset Q^{w}_{j,k}$, then
$$(1+|2^{j-j'}k'-k|)^{-N}\lesssim (1+|w|)^{-N}$$ (see also (\ref{eqn:est1})). We obtain that
$$\begin{array}{rl}
III^{\gamma_{1},\gamma_{2}}_{p,q,Q_{r},m}
&\lesssim\ |Q_{r}|^{\frac{q\gamma_{2}}{n}-\frac{q}{p}}\sum\limits_{j\geq-\log_{2}r}
2^{jq(\gamma_{1}+\frac{n}{2}-\frac{n}{p})}\Big[\int^{r^{2\beta}}_{2^{-2j\beta}}(t2^{2j\beta})^{m}\sum\limits_{(\epsilon,k)\in S^{j}_{r}}e^{-cpt2^{2j\beta}}\\
&\ \sum\limits_{|j-j'|\leq3}\sum\limits_{w\in\mathbb{Z}^{n}}\sum\limits_{(\epsilon',k')\in S^{w,j'}_{j,k}}|a^{\epsilon'}_{j',k'}|^{p}(1+|2^{j-j'}k'-k|)^{-N'}\frac{dt}{t}\Big]^{\frac{q}{p}}.
\end{array}$$
The number of $Q_{j',k'}$ which
are contained in the dyadic cube $ Q_{j,k}^{w}=
2^{8-j}w+\widetilde{Q}_{j,k}$ equals to $2^{n(8+j'-j)}$. On the
other hand, for any dyadic cube $Q_{r}$ with radius $r$, the number
of $Q_{j,k}\subset Q_{r}$ equals to $(2^{j}r)^{n}$. Then the number
of $Q_{j',k'}$ which are contained in the dyadic cube $Q^{w}_{r}$
equals $(2^{8+j'}r)^{n}$. Denote $S^{w,j'}_{r}$ the set of $(\epsilon',k') $ such that $Q_{j',k'}\subset Q^{w}_{r}$.
Finally we have
$$\begin{array}{rl}
III^{\gamma_{1},\gamma_{2}}_{p,q,Q_{r},m}&\lesssim\ |Q_{r}|^{\frac{q\gamma_{2}}{n}-\frac{q}{p}}\sum\limits_{j'\geq-\log_{2}r-3}2^{j'q(\gamma_{1}+\frac{n}{2}-\frac{n}{p})}
\Big\{\int^{r^{2\beta}}_{2^{-2j\beta}}(t2^{2j'\beta})\\
&\quad [\sum\limits_{|w|\leq
2^{n}}+\sum\limits_{|w|>
2^{n}}](1+|w|)^{-N} \sum\limits_{(\epsilon',k')\in
 S^{w,j'}_{r}}e^{-ct2^{2j'\beta}}|a^{\epsilon'}_{j',k'}|^{p}\frac{dt}{t}\Big\}^{\frac{q}{p}}\\
 &=:\ M_{1}+M_{2}.
\end{array}$$
By the definition of $\dot{B}^{\gamma_{1},\gamma_{2}}_{p,q}$, it is easy to see that $M_{1}\lesssim \|f\|_{\dot{B}^{\gamma_{1},\gamma_{2}}_{p,q}}$.
For the term $M_{2}$, we divide the
estimate into two cases.

Case 3.1: $q\le p$. For this case, $j'\geq-\log_{2}r-3$ implies
$(2^{nj'}r^{n})^{N'}\lesssim 1$ and
$$\begin{array}{rl}
M_{2}\lesssim&\ |Q_{r}|^{\frac{q\gamma_{2}}{n}-\frac{q}{p}}\sum\limits_{j'\geq-\log_{2}r-3}2^{qj'(\gamma_{1}+\frac{n}{2}-\frac{n}{p})}
\sum\limits_{w\in\mathbb{Z}^{n}}(1+|w|)^{-\frac{qN'}{p}}\\
&\Big[\int^{r^{2\beta}}_{2^{-2j\beta}}(t2^{2j'\beta})^{m}\sum\limits_{(\epsilon',k')\in
S^{w,j'}_{r}}e^{-cpt2^{2j'\beta}}|a^{\epsilon}_{j',k'}|^{p}(2^{nj'}|Q|)^{-N'}\frac{dt}{t}\Big]^{\frac{q}{p}}\\
\lesssim&\ \|f\|_{\dot{B}^{\gamma_{1},\gamma_{2}}_{p,q}}.
\end{array}$$

Case 3.2: $q>p$. For this case, by H\"older's inequality and $j\sim j'$ we obtain
$$\begin{array}{rl}
&\Big[\int^{r^{2\beta}}_{2^{-2j\beta}}(t2^{2j\beta})^{m}\sum\limits_{|w|>2^{n}}\sum\limits_{(\epsilon',k')\in
S^{w,j'}_{r}}e^{-cpt2^{2j\beta}}(1+|w|)^{-N}|a^{\epsilon'}_{j',k'}|^{p}(2^{nj'}|Q|)^{-N'}\frac{dt}{t}\Big]^{\frac{q}{p}}\\
&\lesssim\sum\limits_{|w|>2^{n}}(1+|w|)^{-\frac{qN'}{p}}\Big(\int^{r^{2\beta}}_{2^{-2j\beta}}(t2^{2j'\beta})^{m}
\sum\limits_{(\epsilon',k')\in
S^{w,j'}_{r}}e^{-cpt2^{2j'\beta}}|a^{\epsilon'}_{j',k'}|^{p}\frac{dt}{t}\Big)^{\frac{q}{p}}.
\end{array}$$
The rest of the argument is similar to that of Case 3.1, and so omitted.

{\bf Situation 4}: $K^{\beta}_{t}\ast f\in \B^{\gamma_{1},\gamma_{2}, IV}_{p,q,m}$.
Because $|j-j'|\leq 1$ and $0<t<2^{-2j\beta}$,  we can obtain
$$\begin{array}{rl}
IV^{\gamma_{1},\gamma_{2}}_{p,q,Q_{r},
m'}
\lesssim&|Q_{r}|^{\frac{q\gamma_{2}}{n}-\frac{q}{p}}\sum\limits_{j\geq-\log_{2}r}
2^{jq(\gamma_{1}+\frac{n}{2}-\frac{n}{p})} \Big[\sum\limits_{|w|\leq
2^{n}}\frac{1}{(1+|w|)^{N}}\sum\limits_{(\epsilon',k')\in
S^{w,j'}_{r}}|a^{\epsilon'}_{j',k'}|^{p}\Big]^{\frac{q}{p}}\\
+&|Q_{r}|^{\frac{q\gamma_{2}}{n}-\frac{q}{p}}\sum\limits_{j\geq-\log_{2}r}
2^{jq(\gamma_{1}+\frac{n}{2}-\frac{n}{p})} \Big[\sum\limits_{|w|>
2^{n}}\frac{1}{(1+|w|)^{N}}\sum\limits_{(\epsilon',k')\in
S^{w,j'}_{r}}\frac{|a^{\epsilon'}_{j',k'}|^{p}}{(2^{nj'}|Q|)^{N'}}\Big]^{\frac{q}{p}}.
\end{array}$$

Case 4.1: $q\le p$. For this case, by the $\alpha-$triangle inequality we have
$$\begin{array}{rl}
IV^{\gamma_{1},\gamma_{2}}_{p,q,Q_{r},m'}\lesssim&\sum\limits_{w\in\mathbb{Z}^{n}}\frac{|Q|^{\frac{q\gamma_{2}}{n}-\frac{q}{p}}}{(1+|w|)^{\frac{qN}{p}}}
\sum\limits_{j'\geq-\log_{2}r-1}
2^{qj'(\gamma_{1}+\frac{n}{2}-\frac{n}{p})}
\Big(\sum\limits_{(\epsilon',k')\in
S^{w,j'}_{r}}|a^{\epsilon'}_{j',k'}|^{p}\Big)^{\frac{q}{p}}\\
\lesssim&\|f\|_{\dot{B}^{\gamma_{1},\gamma_{2}}_{p,q}}.
\end{array}$$

Case 4.2: $q>p$. Using H\"older's inequality we have
$$\begin{array}{rl}
IV^{\gamma_{1},\gamma_{2}}_{p,q,Q_{r},m'}\lesssim&\sum\limits_{w\in\mathbb{Z}^{n}}\frac{|Q|^{\frac{q\gamma_{2}}{n}-\frac{q}{p}}}{(1+|w|)^{N}}
\sum\limits_{j'\geq-\log_{2}r-1}
2^{qj'(\gamma_{1}+\frac{n}{2}-\frac{n}{p})}
\Big(\sum\limits_{(\epsilon',k')\in
S^{w,j'}_{r}}|a^{\epsilon'}_{j',k'}|^{p}\Big)^{\frac{q}{p}}\\
\lesssim&\|f\|_{\dot{B}^{\gamma_{1},\gamma_{2}}_{p,q}}.
\end{array}$$
This completes the proof of Theorem \ref{th1}.

\end{proof}

We close this section by showing the following continuity of the Riesz transforms acting on the Besov-Q spaces; see also \cite{Al} and \cite{Yang1} for some related results.

\begin{theorem}\label{th4}
For $1<p, q<\infty, \gamma_{1},\gamma_{2}\in \mathbb{R}, m>p$,
and $m'>0$, the Riesz transforms $R_1,R_2,...,R_n$ are continuous on
$\B^{\gamma_{1},\gamma_{2}}_{p, q, m,m'}$.
\end{theorem}
\begin{proof}
For the sake of convenience, we choose the classical Meyer wavelet basis
$\{\Phi^{\epsilon}_{j,k}\}_{(\epsilon,j,k)\in\Lambda_{n}}$. For
any $g(\cdot,\cdot)\in\B^{\gamma_{1},\gamma_{2}}_{p,q,m,m'}$ and $l=1,2,...,n$, we
need to prove $(R_{l}g)(\cdot,\cdot)\in
\B^{\gamma_{1},\gamma_{2}}_{p,q,m,m'}$. Write
$g(t,x)=\sum\limits_{(\epsilon,j,k)\in\Lambda_{n}}g^{\epsilon}_{j,k}(t)\Phi^{\epsilon}_{j,k}(x)$.
Then
$$\begin{array}{rl}
&R_{l}g(t,x)=\sum\limits_{(\epsilon,j,k)\in\Lambda_{n}}g^{\epsilon}_{j,k}(t)R_{l}\Phi^{\epsilon}_{j,k}(x)
=:\sum\limits_{(\epsilon,j,k)\in\Lambda_{n}}b^{\epsilon}_{j,k}(t)\Phi^{\epsilon}_{j,k}(x),
\end{array}$$
where  $b^{\epsilon}_{j,k}(t)$ is defined by
$$\begin{array}{rl}
b^{\epsilon}_{j,k}(t)
&=\ \sum\limits_{(\epsilon',j',k')\in\Lambda_{n}}g^{\epsilon'}_{j',k'}(t)\Big\langle
R_{l}\Phi^{\epsilon'}_{j',k'},\
\Phi^{\epsilon}_{j,k}\Big\rangle=:\ \sum\limits_{|j-j'|\leq1}\sum\limits_{\epsilon',k'}a^{\epsilon,\epsilon'}_{j,k,j',k'}g^{\epsilon'}_{j',k'}(t).
\end{array}$$
Because  $R_{l}$ is a Calder\'on-Zygmund operator, by (\ref{eq7}) we get
$$\begin{array}{rl}
|a^{\epsilon,\epsilon'}_{j,k,j',k'}|\lesssim2^{-|j-j'|(\frac{n}{2}+N_{0})}
\Big(\frac{2^{-j}+2^{-j'}}{2^{-j}+2^{-j'}+|2^{-j}k-2^{-j'}k'|}\Big)^{n+N_{0}}.
\end{array}$$
The rest of the proof is similar to Theorem \ref{th1}. We omit the details.
\end{proof}

\section{Nonlinear terms and their a prior
estimates}\label{sec5}

\subsection{Decompositions of non-linear terms} From now on, let
$$
\begin{cases}
u(t,x)=\sum\limits_{(\epsilon,j,k)\in
\Lambda_n} u^{\epsilon}_{j,k}(t) \Phi^{\epsilon}_{j,k}(x);\\
v(t,x)=\sum\limits_{(\epsilon,j,k)\in \Lambda_n}
v^{\epsilon}_{j,k}(t) \Phi^{\epsilon}_{j,k}(x).
\end{cases}
$$
For $l=1,\cdots, n$, we will derive some inequalities about
$$\begin{array}{rl}
B_{l}(u,v)(t,x)=&\int^{t}_{0} e^{-(t-s)(-\Delta)^{\beta}}
\frac{\partial}{\partial x_{l}}(uv) ds.\end{array}
$$
Here, it is worth pointing out that (\ref{eq:decompose}) gives
$$\begin{array}{rl}
e^{-(t-s)(-\Delta)^{\beta}} \frac{\partial}{\partial
x_{l}}(uv)(s,t,x)=\sum\limits_{j'\in\mathbb{Z}}\sum\limits_{i=1}^{4}I^{i,l}_{j'}(s,t,x),
\end{array}$$
where
$$\begin{array}{rcl}
I^{1,l}_{j'}(u,v)(s,t,x)&=& \sum\limits_{\epsilon',k'}
\sum\limits_{k''}
 u^{\epsilon'}_{j',k'}(s) v^{0}_{j'-3,k''}(s)\\
&&\quad\quad\times\ \ e^{-(t-s)(-\Delta)^{\beta}}
\frac{\partial}{\partial x_{l}}
(\Phi^{\epsilon'}_{j',k'}(x)\Phi^{0}_{j'-3,k''}(x)), \\
I^{2,l}_{j'}(u,v)(s,t,x)&=&\sum\limits_{\epsilon',k'}
\sum\limits_{\epsilon'',k''} u^{\epsilon'}_{j',k'}(s)
v^{\epsilon''}_{j',k''}(s)\\
&&\quad\quad \times\ \ e^{-(t-s)(-\Delta)^{\beta}}
\frac{\partial}{\partial x_{l}}
(\Phi^{\epsilon'}_{j',k'}(x)\Phi^{\epsilon''}_{j',k''}(x)),\\
I^{3,l}_{j'}(u,v)(s,t,x)&=& \sum\limits_{0<|j'- j''|\leq 3}
\sum\limits_{\epsilon',k'} \sum\limits_{\epsilon'',k''}
 u^{\epsilon'}_{j',k'}(s) v^{\epsilon''}_{j'',k''}(s)\\
&&\quad\quad\times\ \ e^{-(t-s)(-\Delta)^{\beta}}
\frac{\partial}{\partial x_{l}}
(\Phi^{\epsilon'}_{j',k'}(x)\Phi^{\epsilon''}_{j'',k''}(x)),\\
I^{4,l}_{j'}(u,v)(s,t,x)&=& \sum\limits_{\epsilon',k'}
\sum\limits_{k''} v^{\epsilon'}_{j',k'}(s) u^{0}_{j'-3,k''}(s)\\
&&\quad\quad\times e^{-(t-s)(-\Delta)^{\beta}}
\frac{\partial}{\partial x_{l}}
(\Phi^{\epsilon'}_{j',k'}(x)\Phi^{0}_{j'-3,k''}(x)).
\end{array}$$
Hence
$$\begin{array}{rl}
B_{l}(u,v)(t,x)
&=:\ \int^{t}_{0}\sum\limits_{j'\in\mathbb{Z}}\sum\limits_{i=1}^{4}I^{i,l}_{j'}(s,t,x)ds
=:\ \sum\limits_{i=1}^{4}\int^{t}_{0}I^{i}_{l}(s,t,x)ds.\\
\end{array}$$
Therefore, we can write
\begin{equation}\label{eq:de}
\begin{array}{rl}
B_{l}(u,v)(t,x) :=& \sum\limits^{4}_{i=1} I^{i}_{l}(u,v)(t,x),
\end{array}
\end{equation}
where
$$\begin{array}{rl}
I^{i}_{l}(u,v)(t,x)=\int^{t}_{0}I^{i}_{l}(s,t,x)ds.
\end{array}$$
In order to estimate the bilinear term $B(u,v)$ in some suitable function spaces on
$\mathbb{R}^{n}$, we are required to decompose the terms $I^{i}_{l}(u,v)(t,x)$,
$i=1,2,\cdots, 4$, respectively.

{\bf Decomposition of $I^{1}_{l}(u,v)(t,x)$.} The term $I^{1}_{l}(u,v)(t,x)$ is decomposed according to two cases.

Case $[I^{1}_{l}]_1$: $t\geq 2^{-2j\beta}$. For this case, we write
$I^1_l(u,v)(t,x)$ as the sum of the following three terms:
$$\begin{array}{rl}
I^{1}_l(u,v)(t,x) =&\!\!\!
\sum\limits_{\epsilon',j',k'} \sum\limits_{k''}
\Big(\int^{2^{-1-2j'\beta}}_{0}+\int^{\frac{t}{2}}_{2^{-1-2j'\beta}}+\int_{\frac{t}{2}}^{t}\Big) \Big\{u^{\epsilon'}_{j',k'}(s)
v^{0}_{j'-3,k''}(s)\\
&\quad\times\ \ e^{-(t-s)(-\Delta)^{\beta}}
\frac{\partial}{\partial x_{l}}
(\Phi^{\epsilon'}_{j',k'}(x)\Phi^{0}_{j'-3,k''}(x))\Big\}ds\\
=:& I^{1,1}_l(u,v)(t,x)+I^{1,2}_l(u,v)(t,x)+I^{1,3}_l(u,v)(t,x).
\end{array}$$

For $i=1,2,3$, denote
$$\begin{array}{rl}
I^{1,i}_l(u,v)(t,x) = \sum\limits_{(\epsilon,j,k)\in \Lambda_n}
a^{\epsilon,i}_{j,k}(t) \Phi^{\epsilon}_{j,k}(x). \end{array}$$

Case $[I^{1}_{l}]_2$: $t<2^{-2j\beta}$. For this case , we denote
$a^{\epsilon,4}_{j,k}(t)=a^{\epsilon}_{j,k}(t)$ and then have
$$\begin{array}{rl}I^1_l(u,v)(t,x) = \sum\limits_{(\epsilon,j,k)\in \Lambda_n}
a^{\epsilon,4}_{j,k}(t) \Phi^{\epsilon}_{j,k}(x).\end{array}$$

{\bf Decomposition of $I^{2}_{l}(u,v)(t,x)$.} The decomposition of
$I^{2}_{l}(u,v)(t,x)$ is made according to two cases.

Case $[I^{2}_{l}]_1$: $t\geq 2^{-2j\beta}$. Naturally, $I^2_l(u,v)(t,x)$ can be divided into
the following three terms:
$$\begin{array}{rl}
I^{2}_l(u,v)(t,x) =&\!\!\! \sum\limits_{j'}
\sum\limits_{\epsilon',k'} \sum\limits_{\epsilon'',k''}
\Big(\int^{2^{-1-2j'\beta}}_{0}+\int^{\frac{t}{2}}_{2^{-1-2j'\beta}}+\int_{\frac{t}{2}}^{t}\Big)
\Big\{u^{\epsilon'}_{j',k'}(s) v^{\epsilon''}_{j',k''}(s)\\
&\quad\times\ \ e^{-(t-s)(-\Delta)^{\beta}}  \frac{\partial}{\partial
x_{l}}
(\Phi^{\epsilon'}_{j',k'}(x)\Phi^{\epsilon''}_{j',k''}(x))\Big\}ds\\
=:&\!\!\! I^{2,1}_l(u,v)(t,x)+I^{2,2}_l(u,v)(t,x)+I^{2,3}_l(u,v)(t,x).
\end{array}$$

Case $[I^{2}_{l}]_2$: $t\leq 2^{-2j\beta}$. This $I^{2}_{l}(u,v)(t,x)$ can be decomposed
into the sum of $II^{4}(u,v)(t,x)$ and $II^{5}(u,v)(t,x)$, where
$$\begin{array}{rl}
I^{2}_l(u,v)(t,x) =&\!\!\! \sum\limits_{j'}
\sum\limits_{\epsilon',k'} \sum\limits_{\epsilon'',k''}
\Big(\int^{2^{-2j'\beta}}_{0}+\int^{t}_{2^{-2j'\beta}}\Big)
\Big\{u^{\epsilon'}_{j',k'}(s) v^{\epsilon''}_{j',k''}(s)\\
&\quad\times e^{-(t-s)(-\Delta)^{\beta}}  \frac{\partial}{\partial
x_{l}}
(\Phi^{\epsilon'}_{j',k'}(x)\Phi^{\epsilon''}_{j',k''}(x))\Big\}ds\\
=:&\!\!\!I^{2,4}_l(u,v)(t,x)+I^{2,5}_l(u,v)(t,x).
\end{array}$$
For $i=1,2,3,4,5$, set
$$\begin{array}{rl}I^{2,i}_l(u,v)(t,x) = \sum\limits_{(\epsilon,j,k)\in \Lambda_n} b^{\epsilon,i}_{j,k}(t)
\Phi^{\epsilon}_{j,k}(x).\end{array}$$

{\bf Decompositions of $I^{3}_{l}(u,v)(t,x)$.} Similarly, we have the following two cases:

Case $[I^{3}_{l}]_1$: $t\geq 2^{-2j\beta}$. This $I^3_l(u,v)(t,x)$ can be divided into
the following three terms:
$$\begin{array}{rl}
I^{3}_l(u,v)(t,x) =&\!\!\! \sum\limits_{0<|j'- j''|\leq 3}
\sum\limits_{\epsilon',k'}
\sum\limits_{\epsilon'',k''}\Big(\int^{2^{-1-2j'\beta}}_{0}+\int^{\frac{t}{2}}_{2^{-1-2j'\beta}}+\int_{\frac{t}{2}}^{t}\Big)
\Big\{u^{\epsilon'}_{j',k'}(s) v^{\epsilon''}_{j'',k''}(s)\\
&\quad\times\ \ e^{-(t-s)(-\Delta)^{\beta}}  \frac{\partial}{\partial
x_{l}}
(\Phi^{\epsilon'}_{j',k'}(x)\Phi^{\epsilon''}_{j'',k''}(x))\Big\}ds\\
=:&\!\!\!I^{3,1}_l(u,v)(t,x)+I^{3,2}_l(u,v)(t,x) +I^{3,3}_l(u,v)(t,x).
\end{array}$$

Case $[I^{3}_{l}]_2$: $t\leq 2^{-2j\beta}$. This $I^{3}_{l}(u,v)(t,x)$ can be decomposed
into the sum of $II^{4}(u,v)(t,x)$ and $II^{5}(u,v)(t,x)$, where
$$\begin{array}{rl}
I^{3}_l(u,v)(t,x) =&\!\!\! \sum\limits_{0<|j'- j''|\leq 3}
\sum\limits_{\epsilon',k'} \sum\limits_{\epsilon'',k''}
\Big(\int^{2^{-2j'\beta}}_{0}+\int^{t}_{2^{-2j'\beta}}\Big)
\Big\{u^{\epsilon'}_{j',k'}(s) v^{\epsilon''}_{j'',k''}(s)\\
&\quad\times\ \ e^{-(t-s)(-\Delta)^{\beta}}  \frac{\partial}{\partial
x_{l}}
(\Phi^{\epsilon'}_{j',k'}(x)\Phi^{\epsilon''}_{j'',k''}(x))\Big\}ds\\
=:&\!\!\!I^{3,4}_l(u,v)(t,x)+I^{3,5}_l(u,v)(t,x).
\end{array}$$
For $i=1,2,3,4,5$, denote
$$\begin{array}{rl}I^{3,i}_l(u,v)(t,x) = \sum\limits_{(\epsilon,j,k)\in \Lambda_n} b^{\epsilon,i}_{j,k}(t)
\Phi^{\epsilon}_{j,k}(x).\end{array}$$

{\bf Decomposition of $I^{4,l}_{j}(u,v)(t,x)$.} It is easy to
see that the terms $I^{1, l}_{j}(u,v)(t,x)$ and
$I^{4,l}_{j}(u,v)(t,x)$ are symmetric associated with $u(t,x)$ and
$v(t,x)$. Hence for $I^{4}_{l}(u,v)$ we have a similar
decomposition.

Case $[I^{4}_{l}]_1$: $t\geq 2^{-2j\beta}$. For this case, we write
$I^{4}_{l}(u,v)(t,x)$ as the sum of the following three terms:
$$\begin{array}{rl}
I^{4}_l(u,v)(t,x) =&\!\!\!
\sum\limits_{\epsilon',j',k'} \sum\limits_{k''}
\Big(\int^{2^{-1-2j'\beta}}_{0}+\int^{\frac{t}{2}}_{2^{-1-2j'\beta}}+\int_{\frac{t}{2}}^{t}\Big)\Big\{ v^{\epsilon'}_{j',k'}(s)
u^{0}_{j'-3,k''}(s)\\
&\quad\quad\times\ \ e^{-(t-s)(-\Delta)^{\beta}}
\frac{\partial}{\partial x_{l}}
(\Phi^{\epsilon'}_{j',k'}(x)\Phi^{0}_{j'-3,k''}(x))\Big\}ds\\
=:&\!\!\!I^{4,1}_l(u,v)(t,x)+I^{4,2}_l(u,v)(t,x)+I^{4,3}_l(u,v)(t,x).
\end{array}$$

For $i=1,2,3$, denote
$$\begin{array}{rl}
I^{4,i}_l(u,v)(t,x) = \sum\limits_{(\epsilon,j,k)\in \Lambda_n}
a^{\epsilon,i}_{j,k}(t) \Phi^{\epsilon}_{j,k}(x). \end{array}$$

Case $[I^{4}_{l}]_2$: $t<2^{-2j\beta}$. For this case, we denote
$a^{\epsilon,4}_{j,k}(t)=a^{\epsilon}_{j,k}(t)$ and then have
$$\begin{array}{rl}I^1_l(u,v)(t,x) = \sum\limits_{(\epsilon,j,k)\in \Lambda_n}
a^{\epsilon,4}_{j,k}(t) \Phi^{\epsilon}_{j,k}(x).\end{array}$$

\subsection{Induced a prior estimates} In the sequel we are about to dominate the above-defined $a^{\epsilon,i}_{j,k}$, $b^{\epsilon,i}_{j,k}$ by $u^{\epsilon'}_{j',k'}$ and $v^{\epsilon''}_{j',k''}$.

\begin{lemma}\label{le6} There is a constant $\tilde{c}>0$ such that:
\item{\rm(i)} For $i=1, 2$,
$$\begin{array}{rl}
|a^{\epsilon,i}_{j,k}(t)|\lesssim &\!\!\!
 2^{\frac{nj}{2}+j} \sum\limits_{|j-j'|\leq 2}
\sum\limits_{\epsilon',k',k''}
\int_{I_{i}}\frac{|u^{\epsilon'}_{j',k'}(s)|}{(1+ |2^{j-j'}k'-k|)^{N}}\frac{| v^{0}_{j'-3,k''}(s)|}{(1+|2^{j-j'+3}k''-k'|)^{N}}
 e^{-\tilde c t2^{2j\beta}}ds,
\end{array}$$
where $I_{1}=[0, 2^{-1-2j'\beta}]$ and $I_{2}=[2^{-1-2j'\beta}, \frac{t}{2}]$.

\item{\rm(ii)} For $i=3, 4$,
$$\begin{array}{rl}
|a^{\epsilon,i}_{j,k}(t)|\lesssim &\!\!\!
2^{\frac{nj}{2}+j} \sum\limits_{|j-j'|\leq 2}
\sum\limits_{\epsilon',k',k''}
\int_{I_{i}}\frac{|u^{\epsilon'}_{j',k'}(s)|}{(1+ |2^{j-j'}k'-k|)^{N}}\frac{| v^{0}_{j'-3,k''}(s)|}{(1+|2^{j-j'+3}k''-k'|)^{N}}
e^{-\tilde c (t-s)2^{2j\beta}}ds,
\end{array}$$
where $I_{3}=[\frac{t}{2}, t]$ and $I_{4}=[0, t]$.

\end{lemma}
\begin{proof}
$$\begin{array}{rl}
a^{\epsilon,1}_{j,k}(t)=&\!\!\!\Big\langle I^{1,1}_{l}(u,v),
\Phi^{\epsilon}_{j,k}\Big\rangle\\
=&\!\!\!\sum\limits_{\epsilon',k',k''}~\sum\limits_{|j-j'|\leq2}\int^{2^{-1-2j'\beta}}_{0}\Big\{u^{\epsilon'}_{j',k'}(s)v^{0}_{j'-3,k''}(s)\\
&\quad\quad\Big\langle
e^{-(t-s)(-\Delta)^{\beta}}\frac{\partial}{\partial
x_{l}}(\Phi^{\epsilon'}_{j',k'}\Phi^{0}_{j'-3,k''}),\
\Phi^{\epsilon}_{j,k}\Big\rangle \Big\}ds.
\end{array}$$
The Fourier transform gives
$$\begin{array}{rl}
&\Big\langle
e^{-(t-s)(-\Delta)^{\beta}}\frac{\partial}{\partial
x_{l}}(\Phi^{\epsilon'}_{j',k'}\Phi^{0}_{j'-3,k''}),\
\Phi^{\epsilon}_{j,k}\Big\rangle\\
&=\int e^{-(t-s)|\xi|^{2\beta}}\xi_{l}\widehat{(\Phi^{\epsilon'}_{j',k'}\Phi^{0}_{j'-3,k''})}(\xi)
2^{-jn/2}e^{-i2^{-j}k\xi}\widehat{\Phi^{\epsilon}}(2^{-j}\xi)d\xi\\
&=\int e^{-(t-s)|\xi|^{2\beta}}\xi_{l}e^{-i2^{-j'}k'\xi}\Big[\int e^{ik'\eta}\widehat{\Phi^{\epsilon'}}(2^{-j'}\xi-\eta)e^{-8ik''\eta}\widehat{\Phi^{0}}(8\eta)d\eta\Big]\\
&\quad\times 2^{-jn/2}e^{-i2^{-j}k\xi}\widehat{\Phi^{\epsilon}}(2^{-j}\xi)d\xi.
\end{array}$$
Because $0<s<2^{-1-2j'\beta}$, we can see $(t-s)\sim t$. Hence
$$\begin{array}{rl}
&\Big\langle
e^{-(t-s)(-\Delta)^{\beta}}\frac{\partial}{\partial
x_{l}}(\Phi^{\epsilon'}_{j',k'}\Phi^{0}_{j'-3,k''}),\
\Phi^{\epsilon}_{j,k}\Big\rangle\\
&=2^{jn/2+j}\int e^{-(t-s)2^{2j\beta}|\xi|^{2\beta}}\xi_{l}e^{-i(k-2^{j-j'}k')\xi}\Big[\int \widehat{\Phi^{\epsilon'}}(2^{j-j'}\xi-\eta)\widehat{\Phi^{0}}(8\eta)\\
&\quad\times e^{-i(k-8k'')\eta}d\eta\Big] \widehat{\Phi^{\epsilon}}(\xi)d\xi\\
&=2^{jn/2+j}\int e^{-t2^{2j\beta}|\xi|^{2\beta}}\xi_{l}e^{-i(k-2^{j-j'}k')\xi}\Big[\int \widehat{\Phi^{\epsilon'}}(2^{j-j'}\xi-\eta)\widehat{\Phi^{0}}(8\eta)\\
&\quad\times e^{-i(k-8k'')\eta}d\eta\Big] \widehat{\Phi^{\epsilon}}(\xi)d\xi.
\end{array}$$

{\bf Situation I: } We first consider the case: $|2^{j-j'}k'-k|\leq 2$. We can see $$(1+|2^{j-j'}k'-k|)^{-N}\gtrsim 1.$$
 Under this situation, we divide the argument into two cases.

{\it Case 1:} $|k'-8k''|\leq 2$. For any positive integer $N$, $$(1+|2^{j-j'+3}k'-k''|)^{-N}\gtrsim 1.$$
On the other hand, the support of $\widehat{\Phi^{\epsilon}}(\xi)$ is a ring. A direct computation derives
$$\begin{array}{rl}
&\Big|\Big\langle
e^{-(t-s)(-\Delta)^{\beta}}\frac{\partial}{\partial
x_{l}}(\Phi^{\epsilon'}_{j',k'}\Phi^{0}_{j'-3,k''}),\
\Phi^{\epsilon}_{j,k}\Big\rangle\Big|\\
&\lesssim e^{-\tilde c t2^{2j\beta}}2^{jn/2+j}  (1+ |2^{j-j'}k'-k|)^{-N}
(1+|2^{j-j'+3}k''-k'|)^{-N}.
\end{array}$$

{\it Case 2:} $|k'-8k''|\geq 2$. Denote by $l_{j_{0}}$ the largest component of $k'-8k''$. We have
$$\begin{array}{rl}
&\Big|\int \widehat{\Phi^{\epsilon'}}(2^{j-j'}\xi-\eta)\widehat{\Phi^{0}}(8\eta) e^{-i(k-8k'')\eta}d\eta\Big|\\
&\lesssim\frac{1}{(1+|k'-8k''|)^{N}}\Big|\int\widehat{\Phi^{\epsilon'}}(2^{j-j'}\xi-\eta)\widehat{\Phi^{0}}(8\eta) (\frac{1}{i}\partial_{\eta_{i_{0}}})^{N}(e^{-i(k-8k'')\eta})d\eta\Big|\\
&\lesssim\frac{1}{(1+|k'-8k''|)^{N}}\Big|\int\sum\limits_{l=0}^{N}C^{l}_{N}\partial_{\eta_{i_{0}}}^{l}(\widehat{\Phi^{\epsilon'}}(2^{j-j'}\xi-\eta))
\partial_{\eta_{i_{0}}}^{N-l}(\widehat{\Phi^{0}}(8\eta)) e^{-i(k-8k'')\eta}d\eta\Big|\\
&\lesssim\frac{1}{(1+|k'-8k''|)^{N}},
\end{array}$$
which gives
$$\begin{array}{rl}
&\Big|\Big\langle
e^{-(t-s)(-\Delta)^{\beta}}\frac{\partial}{\partial
x_{l}}(\Phi^{\epsilon'}_{j',k'}\Phi^{0}_{j'-3,k''}),\
\Phi^{\epsilon}_{j,k}\Big\rangle\Big|\\
&\lesssim 2^{jn/2+j}e^{-\tilde c t2^{2j\beta}}  (1+ |2^{j-j'}k'-k|)^{-N}
(1+|2^{j-j'+3}k''-k'|)^{-N}.
\end{array}$$

{\bf Situation II: } We then consider the case: $|2^{j-j'}k'-k|\geq 2$. We still divide the discussion into the following two cases.

{\it Case 3: } $|k'-8k''|\leq 2$. Denote by $k_{i_{0}}$ the largest component of $2^{j-j'}k'-k$. Then
$$(1+|k_{i_{0}}|)^{N}\sim (1+|2^{j-j'}k'-k|)^{N}.$$
This fact implies
$$\begin{array}{rl}
&\Big|\Big\langle
e^{-(t-s)(-\Delta)^{\beta}}\frac{\partial}{\partial
x_{l}}(\Phi^{\epsilon'}_{j',k'}\Phi^{0}_{j'-3,k''}),\
\Phi^{\epsilon}_{j,k}\Big\rangle\Big|\\
&\lesssim\frac{2^{jn/2+j}}{(1+|2^{j-j'}k'-k|)^{N}}\Big|\int (\frac{1}{i}\partial_{\xi_{i_{0}}})^{N}(e^{-i(k-2^{j-j'}k')\xi})e^{-t2^{2j\beta}|\xi|^{2\beta}}\xi_{l}\\
&\quad\times\Big[\int \widehat{\Phi^{\epsilon'}}(2^{j-j'}\xi-\eta)\widehat{\Phi^{0}}(8\eta)
 e^{-i(k-8k'')\eta}d\eta\Big] \widehat{\Phi^{\epsilon}}(\xi)d\xi\Big|\\
&\lesssim\frac{2^{jn/2+j}}{(1+|2^{j-j'}k'-k|)^{N}}\Big|\int e^{-i(k-2^{j-j'}k')\xi}\sum\limits^{N}_{l=0}C^{l}_{N}\partial_{\xi_{i_{0}}}^{l}(e^{-t2^{2j\beta}|\xi|^{2\beta}}\xi_{l})\\
&\quad\times\partial_{\xi_{i_{0}}}^{N-l}\Big(\int \widehat{\Phi^{\epsilon'}}(2^{j-j'}\xi-\eta)\widehat{\Phi^{0}}(8\eta)
 e^{-i(k-8k'')\eta}d\eta\Big) \widehat{\Phi^{\epsilon}}(\xi)d\xi\Big|.\\
\end{array}$$
In the above and below, $C^l_N$ stands for the binomial coefficient indexed by $N$ and $l$.
Because the support of $\widehat{\Phi^{\epsilon}}(\xi)$ is a ring, there exists a small constant $c>0$ such that
$$|\partial_{\xi_{i_{0}}}^{l}(e^{-t2^{2j\beta}|\xi|^{2\beta}}\xi_{l})|\lesssim e^{-ct2^{2j\beta}}.$$
Consequently, we have a constant $\tilde{c}>0$ such that
$$\begin{array}{rl}
&\Big|\Big\langle
e^{-(t-s)(-\Delta)^{\beta}}\frac{\partial}{\partial
x_{l}}(\Phi^{\epsilon'}_{j',k'}\Phi^{0}_{j'-3,k''}),\
\Phi^{\epsilon}_{j,k}\Big\rangle\Big|\\
&\lesssim\frac{2^{jn/2+j}}{(1+|2^{j-j'}k'-k|)^{N}}\Big|\int e^{-i(k-2^{j-j'}k')\xi}\sum\limits^{N}_{l=0}C^{l}_{N}\partial_{\xi_{i_{0}}}^{l}(e^{-t2^{2j\beta}|\xi|^{2\beta}}\xi_{l})\\
&\quad\times\partial_{\xi_{i_{0}}}^{N-l}\Big(\int \widehat{\Phi^{\epsilon'}}(2^{j-j'}\xi-\eta)\widehat{\Phi^{0}}(8\eta)
 e^{-i(k-8k'')\eta}d\eta\Big) \widehat{\Phi^{\epsilon}}(\xi)d\xi\Big|\\
&\lesssim e^{-\tilde c t2^{2j\beta}} 2^{jn/2+j} (1+ |2^{j-j'}k'-k|)^{-N}
(1+|2^{j-j'+3}k''-k'|)^{-N}.
\end{array}$$

{\it Case 4:} $|k'-8k''|\geq 2$. In a similar manner to treat Case 3, we denote by $k_{i_{0}}$ the largest component of $2^{j-j'}k'-k$, and then obtain
$$\begin{array}{rl}
&\Big|\Big\langle
e^{-(t-s)(-\Delta)^{\beta}}\frac{\partial}{\partial
x_{l}}(\Phi^{\epsilon'}_{j',k'}\Phi^{0}_{j'-3,k''}),\
\Phi^{\epsilon}_{j,k}\Big\rangle\Big|\\
&\lesssim\frac{2^{jn/2+j}}{(1+|2^{j-j'}k'-k|)^{N}}\int \sum\limits^{N}_{l=0}C^{l}_{N}\Big|\partial_{\xi_{i_{0}}}^{l}(e^{-t2^{2j\beta}|\xi|^{2\beta}}\xi_{l})\Big|\\
&\quad\times\Big|\int \partial_{\xi_{i_{0}}}^{N-l}\Big(\widehat{\Phi^{\epsilon'}}(2^{j-j'}\xi-\eta)\Big)\widehat{\Phi^{0}}(8\eta)
 e^{-i(k-8k'')\eta}d\eta\Big| |\widehat{\Phi^{\epsilon}}(\xi)|d\xi.
\end{array}$$
As in Case 2, upon choosing $l_{j_{0}}$ as the largest component of $k'-8k’’$, applying an integration-by-parts, and utilizing the fact that $\widehat{\Phi^{\epsilon}}$ is supported on a ring, we can get a constant $\tilde{c}>0$ such that
$$\begin{array}{rl}
&\Big|\Big\langle
e^{-(t-s)(-\Delta)^{\beta}}\frac{\partial}{\partial
x_{l}}(\Phi^{\epsilon'}_{j',k'}\Phi^{0}_{j'-3,k''}),\
\Phi^{\epsilon}_{j,k}\Big\rangle\Big|\\
&\lesssim\frac{2^{jn/2+j}}{(1+|2^{j-j'}k'-k|)^{N}(1+|2^{j-j'+3}k''-k'|)^{N}}\int \sum\limits^{N}_{l=0}C^{l}_{N}\Big|\partial_{\xi_{i_{0}}}^{l}(e^{-t2^{2j\beta}|\xi|^{2\beta}}\xi_{l})\Big|\\
&\quad\times\Big|\int \partial_{\xi_{i_{0}}}^{N-l}\Big(\widehat{\Phi^{\epsilon'}}(2^{j-j'}\xi-\eta)\Big)\widehat{\Phi^{0}}(8\eta)
 (\frac{1}{i}\partial_{\eta_{i_{0}}})^{N}\Big(e^{-i(k'-8k'')\eta}\Big)d\eta\Big| |\widehat{\Phi^{\epsilon}}(\xi)|d\xi\\
&\lesssim  e^{-\tilde c t2^{2j\beta}} 2^{jn/2+j} (1+ |2^{j-j'}k'-k|)^{-N}
(1+|2^{j-j'+3}k''-k'|)^{-N}.
\end{array}$$
This completes the estimate of $a^{\epsilon,1}_{j,k}(t)$. The estimate of $a^{\epsilon,i}_{j,k}(t)$ can be obtained similarly.
\end{proof}

Using the same method, we can obtain the following estimates for $b^{\epsilon,i}_{j,k}(t), i=1,2,3,4,5$.
\begin{lemma}\label{le7}There is a constant $\tilde{c}>0$ such that:

\item{\rm (i)} For $i=1, 2$,
$$\begin{array}{rl}
|b^{\epsilon, i}_{j,k}(t)|\lesssim &\!\!\!
2^{\frac{nj}{2}+j} \sum\limits_{j\leq j'+ 2}
\sum\limits_{\epsilon',k',\epsilon'',k''}
\int_{I_{i}}\frac{|u^{\epsilon'}_{j',k'}(s)|}{(1+ |2^{j-j'}k'-k|)^{N}}\frac{| v^{\epsilon''}_{j',k''}(s)|}{(1+|2^{j-j'}k''-k'|)^{N}}
e^{-\tilde c t2^{2j\beta}}ds,
\end{array}$$
where $I_{1}=[0, 2^{-1-2j'\beta}]$ and $I_{2}=[2^{-1-2j'\beta}, \frac{t}{2}]$.

\item{\rm (ii)} For $i=3, 4, 5$,
$$\begin{array}{rl}
|b^{\epsilon,i}_{j,k}(t)|\lesssim &\!\!\!
2^{\frac{nj}{2}+j} \sum\limits_{j\leq j'+ 2}
\sum\limits_{\epsilon',k',\epsilon'',k''}
\int^{t}_{\frac{t}{2}}\frac{|u^{\epsilon'}_{j',k'}(s)|}{(1+ |2^{j-j'}k'-k|)^{N}}\frac{| v^{\epsilon''}_{j',k''}(s)|}{(1+|2^{j-j'}k''-k'|)^{N}}
e^{-\tilde c (t-s)2^{2j\beta}}ds,\\
\end{array}$$
where $I_{3}=[\frac{t}{2}, t]$, $I_{4}=[0, 2^{-2j'\beta}]$ and $I_{5}=[2^{-2j'\beta}, t]$.

\end{lemma}

Let $Q_{j,k}$ and
$Q_{j',k'}$ be two dyadic cubes, and for $w\in\mathbb{Z}^{n}$ denote by
$Q^{w}_{j,k}$ the dyadic cube $\widetilde{Q}_{j,k}+2^{8-j}w$, where $\widetilde{Q}_{j,k}$ denotes the dyadic cube containing $Q_{j,k}$ with side length $2^{8-j}$. The forthcoming lemmas can be deduced from the Cauchy-Schwartz inequality.

\begin{lemma}\label{inequality1}
\item{\rm (i)}
 For $j, j'\in \mathbb{Z}$ and $w,k,k'\in \mathbb{Z}^{n}$, if
$Q_{j',k'}\subset Q_{j,k}^{w}$, then
\begin{equation}\label{eqn:est1} (1+ |2^{j-j'}k'-k|)^{-N} \lesssim(1+|w|)^{-N}.
\end{equation}
\item{\rm (ii)}
Let $0<j'-j''\leq 3$, $j\leq j'+5$ and $|w-w'|> 2^{n}$. If
$Q_{j',k'}\subset Q_{j,k}^{w}$ and $Q_{j'',k''}\subset
Q_{j,k}^{w'}$, then
\begin{equation}\label{eqn:est2}
(1+|2^{j'-j''}k''-k'|)^{-N}\lesssim 2^{N(j-j')} (1+|w-w'|)^{-N}.
\end{equation}
\end{lemma}

\begin{lemma}\label{inequality3} Let $Q_{j,k}$ be a dyadic cube with radius $2^{-j}$. For
$w\in\mathbb{Z}^{n}$, set $Q^{w}_{j,k}$ be the dyadic cube
$2^{8-j}w+\widetilde{Q}_{j,k}$. Then
\begin{equation}\label{eqn:est3}
\begin{array}{rl}
&\sum\limits_{\epsilon',k'}\sum\limits_{\epsilon'',k''}|u^{\epsilon'}_{j',k'}(s)||v^{\epsilon''}_{j',k''}(s)|^{p-1}
(1+|2^{j-j'}k'-k|)^{-8N}(1+|k'-k''|)^{-8N}\\
&\lesssim \sum\limits_{w\in\mathbb{Z}^{n}}\sum\limits_{w'\in\mathbb{Z}^{n}}(1+|w|)^{-N}(1+|w'|)^{-N}\\
&\ \times\Big(\sum\limits_{(\epsilon',k')\in S^{w,j'}_{
j,k}}|u^{\epsilon'}_{j',k'}(s)|^{p}\Big)^{\frac{1}{p}}\Big(\sum\limits_{(\epsilon'',k'')\in S^{w',j'}_{j,k}}|v^{\epsilon''}_{j',k''}(s)|^{p} \Big)^{\frac{1}{p'}}.
\end{array}
\end{equation}
\end{lemma}

\begin{lemma}\label{inequality4} Let $Q_{j,k}$ be a dyadic cube with radius $2^{-j}$. For
$w\in\mathbb{Z}^{n}$, denote by $Q^{w}_{j,k}$ the dyadic cube
$2^{8-j}w+\widetilde{Q}_{j,k}$. If $\delta>0$ is small
enough, then
\begin{equation}\label{eqn:est5}
\begin{array}{rl}
&\sum\limits_{(\epsilon,k)\in S^{j}_{r}} \Big\{ \sum\limits_{j\leq j'+5}
\sum\limits_{w\in\mathbb{Z}^{n}}(1+|w|)^{-N}
\Big(\sum\limits_{(\epsilon',k')\in S^{w,j'}_{j,k}}
|a^{\epsilon}_{j',k'}|^{p}\Big)^{\frac{1}{p}}\Big\} ^{p}\\
&\lesssim\ \sum\limits_{j\leq j'+5} 2^{\delta (j'-j)}
\sum\limits_{w\in\mathbb{Z}^{n}}(1+|w|)^{-N}
\sum\limits_{(\epsilon',k')\in S^{w,j}_{r}} |a^{\epsilon}_{j',k'}|^{p}.
\end{array}
\end{equation}
\end{lemma}

Here, it is worth mentioning that the proof of Lemma \ref{inequality4} needs also the following fact: for fixed $j$, the number of $Q_{j',k'}$ which
are contained in the dyadic cube $ Q_{j,k}^{w}=
2^{8-j}w+\widetilde{Q}_{j,k}$ equals to $2^{n(8+j'-j)}$. On the
other hand, for any dyadic cube $Q_{r}$ with radius $r$, the number
of $Q_{j,k}\subset Q_{r}$ equals to $(2^{j}r)^{n}$. Then the number
of $Q_{j',k'}$ which are contained in the dyadic cube $Q^{w}_{r}$
equals $(2^{8+j'}r)^{n}$. In the proof of the main lemmas in
Sections \ref{sec7} \& \ref{sec8}, we will use this fact again.

\begin{lemma}
Let $Q_{j,k}$ be a dyadic cube with radius $2^{-j}$. For
$w\in\mathbb{Z}^{n}$, denote by $Q^{w}_{j,k}$ the dyadic cube
$2^{8-j}w+\widetilde{Q}_{j,k}$. If $j<j'+2$, then
\begin{equation}\label{eqn:est6}
\begin{array}{rl}
&\sum\limits_{\epsilon',k'}\sum\limits_{\epsilon'',k''}|u^{\epsilon'}_{j',k'}(s)||v^{\epsilon''}_{j',k''}(s)|
(1+|2^{j'-j}k'-k|)^{-N}(1+|k'-k''|)^{-N}\\
&\lesssim\ \sum\limits_{w,w'\in\mathbb{Z}^{n}}(1+|w|)^{-N}(1+|w-w'|)^{-N}2^{n(j'-j)(1-\frac{2}{p})}\\
&\quad\Big(\sum\limits_{(\epsilon',k')\in S^{w,j'}_{j,k}}|u^{\epsilon'}_{j',k'}(s)|^{p}\Big)^{\frac{1}{p}}
\Big(\sum\limits_{(\epsilon'',k'')\in S^{w',j'}_{j,k}}|v^{\epsilon''}_{j',k''}(s)|^{p}\Big)^{\frac{1}{p}}.
\end{array}
\end{equation}
\end{lemma}

\section{Proof of the main theorem}\label{sec6}

By Picard's contraction principle and Theorems \ref{th1} \& \ref{th4}, it is enough to verify that
the bilinear operator
\begin{equation}\nonumber
\begin{array}{rl}
&B(u,v)= \int^{t}_{0} e^{-(t-s)(-\Delta)^{\beta}}
\mathbb{P}\nabla\cdot (u\otimes v) ds\end{array}
\end{equation}
 is bounded from $(\B^{\gamma_{1}, \gamma_{2}}_{p, q, m,m'
})^{n}\times (\B^{\gamma_{1}, \gamma_{2}}_{p, q, m,m' })^{n}$
to $(\B^{\gamma_{1}, \gamma_{2}}_{p, q, m,m' })^{n}$. To do so, let
$$\begin{array}{rl}
&B_{l}(u,v)= \int^{t}_{0} e^{-(t-s)(-\Delta)^{\beta}}
\frac{\partial}{\partial x_{l}}(uv) ds \end{array}$$
and
$$\begin{array}{rl}
&B_{l,l',l''}(u,v)= R_{l}R_{l'} \int^{t}_{0}
e^{-(t-s)(-\Delta)^{\beta}} \frac{\partial}{\partial x_{l''}}(uv)
ds.\end{array}$$
We need to prove that all $B_{l}(u,v)$,
$B_{l,l',l''}(u,v)$ are bounded from $\B^{\gamma_{1},
\gamma_{2}}_{p, q, m,m' }\times \B^{\gamma_{1}, \gamma_{2}}_{p, q,
m, m' }$ to $\B^{\gamma_{1}, \gamma_{2}}_{p, q, m,m' }$. Because
$R_{l'},\ l'=1,\cdots,n$ are bounded on $\B^{\gamma_{1},
\gamma_{2}}_{p, q, m, m' }$, we only  consider the boundedness of $B_{l}(u,v)$. By
(\ref{eq:de}), if
$$\begin{array}{rl}
&u(t,x)=\sum\limits_{(\epsilon,j,k)\in \Lambda_n}
u^{\epsilon}_{j,k}(t) \Phi^{\epsilon}_{j,k}(x)\ \ \text{ and }
\ \ v(t,x)=\sum\limits_{(\epsilon,j,k)\in \Lambda_n}
v^{\epsilon}_{j,k}(t) \Phi^{\epsilon}_{j,k}(x),
\end{array}$$
then
$$\begin{array}{rl}
B_{l}(u,v)(t,x) = \sum\limits^{4}_{i=1} I^{i}_{l}(u,v)(t,x),
\end{array}$$
where
the terms $I^{i}_{l}(u,v)(t,x),\ i=1,2,\cdots,4$ are defined in
Subsection 4.1.

It is not hard to see that the argument for $I^4_l(u,v)(t,x)$ is similar
to that for $I^1_l(u,v)(t,x)$. Also the treatments of $I^3_l(u,v)(t,x)$
is similar to that of $I^2_l(u,v)(t,x)$. So,
we are only required to show that the following functions
$$\begin{array}{rl}
&(t,x)\mapsto I^1_l(u,v)(t,x) = \sum\limits_{(\epsilon,j,k)\in \Lambda_n}
a^{\epsilon}_{j,k}(t) \Phi^{\epsilon}_{j,k}(x)
\end{array}$$
 and
$$\begin{array}{rl}
&(t,x)\mapsto I^2_l(u,v)(t,x) = \sum\limits_{(\epsilon,j,k)\in \Lambda_n}
b^{\epsilon}_{j,k}(t) \Phi^{\epsilon}_{j,k}(x)
\end{array}$$
belong to $\B^{\gamma_{1},\gamma_{2}}_{p, q, m,m' }$. By the
decompositions of non-linear terms obtained in Subsection 4.1, it amounts to verifying that the following functions
 $$
 \begin{array}{rl}
 (t,x)\mapsto
 \sum\limits_{(\epsilon,j,k)\in\Lambda_{n}}a^{\epsilon,i}_{j,k}(t)\Phi^{\epsilon}_{j,k}(x), \ i=1,2,3,4\end{array}$$
 and
$$\begin{array}{rl}(t,x)\mapsto\sum\limits_{(\epsilon,j,k)\in\Lambda_{n}}b^{\epsilon,i}_{j,k}(t)\Phi^{\epsilon}_{j,k}(x),\ i=1,2,3,4,5\end{array}$$
are members of $\B^{\gamma_{1}, \gamma_{2}}_{p, q, m,m' }$. The demonstration will be concluded by proving the following two lemmas:

\begin{lemma}\label{lem53}
If $(\beta, p,q, \gamma_1, \gamma_2, m,m')$ satisfies the conditions of Theorem \ref{mthmain}
and $u,v\in \B^{\gamma_{1}, \gamma_{2}}_{p, q, m,m' }$, then

\item{\rm(i)} For $ i=1, 2, 3$, the function $(t,x)\mapsto \sum\limits_{(\epsilon,j,k)\in \Lambda_n}
a^{\epsilon,i}_{j,k}(t) \Phi^{\epsilon}_{j,k}(x)$ is in
$\B^{\gamma_{1},\gamma_{2},I}_{p, q, m }$;
\item{\rm(ii)} For $ i=1, 2,
3,$ the function $(t,x)\mapsto \sum\limits_{(\epsilon,j,k)\in \Lambda_n} a^{\epsilon,i}_{j,k}(t)
\Phi^{\epsilon}_{j,k}(x)$ is in $\B^{\gamma_{1},\gamma_{2},III}_{p, q, m
};$
\item{\rm(iii)} The function $(t,x)\mapsto\sum\limits_{(\epsilon,j,k)\in \Lambda_n}
a^{\epsilon,4}_{j,k}(t) \Phi^{\epsilon}_{j,k}(x)$ is in $\B^{\gamma_{1},\gamma_{2},II}_{p,q } \bigcap\B^{\gamma_{1},\gamma_{2},IV}_{p, q, m' }.$
\end{lemma}

\begin{lemma}\label{lem54}
If $(\beta, p,q, \gamma_1, \gamma_2, m,m')$ satisfies the conditions of Theorem \ref{mthmain}
and $u,v\in \B^{\gamma_{1}, \gamma_{2}}_{p, q, m,m' }$, then
\item{\rm (i)} For $i=1, 2, 3$,  the function
$(t,x)\mapsto \sum\limits_{(\epsilon,j,k)\in
\Lambda_n} b^{\epsilon,i}_{j,k}(t) \Phi^{\epsilon}_{j,k}(x)$ are in $\B^{\gamma_{1},\gamma_{2},I}_{p, q, m}$;
\item{\rm (ii)} For $i=1, 2, 3$,  the function $(t,x)\mapsto \sum\limits_{(\epsilon,j,k)\in \Lambda_n} b^{\epsilon,i}_{j,k}(t)
\Phi^{\epsilon}_{j,k}(x)$ is in $\B^{\gamma_{1},\gamma_{2},III}_{p, q, m }$;

\item{\rm (iii)} For $i=4, 5$, the function $(t,x)\mapsto\sum\limits_{(\epsilon,j,k)\in \Lambda_n} b^{\epsilon,i}_{j,k}(t)
\Phi^{\epsilon}_{j,k}(x)$ is in $\B^{\gamma_{1},\gamma_{2},II}_{p, q}$;

\item{\rm (iv)} For $i=4, 5$, the function $(t,x)\mapsto\sum\limits_{(\epsilon,j,k)\in \Lambda_n} b^{\epsilon,i}_{j,k}(t)
\Phi^{\epsilon}_{j,k}(x)$ is in $\B^{\gamma_{1},\gamma_{2},IV}_{p, q, m' }$.
\end{lemma}

\section{Proof of Lemma \ref{lem53}}\label{sec7}

\subsection {The setting (i)}
 For $i=1,2,3$, define
$$\begin{array}{rl}
I^{m,i}_{a, Q_{r}}(t)
&=\ |Q_{r}|^{\frac{q\gamma_{2}}{n}-\frac{q}{p}}\sum\limits_{j\geq\max\{-\log_{2}r,
-\frac{\log_{2}t}{2\beta}\}}2^{qj(\gamma_{1}+\frac{n}{2}-\frac{n}{p})}
\Big[\sum\limits_{(\epsilon,k)\in
S^{j}_{r}}|a^{\epsilon,i}_{j,k}(t)|^{p}(t2^{2j\beta})^{m}\Big]^{\frac{q}{p}}.
\end{array}$$
According to the relation between $2^{-2j\beta}$ and $t$, we divide
the proof into three cases.
The proofs of $\sum\limits_{(\epsilon,j,k)}a^{\epsilon,i}_{j,k}(t)\Phi^{\epsilon}_{j,k}(x), i=1,2,3$ are similar. For simplicity, we only prove
$$
\begin{array}{rl}
&\text{Case 7.1: }\ (t,x)\mapsto\sum\limits_{(\epsilon,j,k)}a^{\epsilon,1}_{j,k}(t)\Phi^{\epsilon}_{j,k}(x)
\ \ \hbox{is\ in}\ \ \B^{\gamma_{1},\gamma_{2},
I}_{p,q,m}.\end{array}
$$
Without loss of generality, we may assume $\|u\|_{\B^{\gamma_{1},\gamma_{2}}_{p,q,m,m'}}=\|v\|_{\B^{\gamma_{1},\gamma_{2}}_{p,q,m,m'}}=1$.
Because
$v\in\B^{\gamma_{1},\gamma_{2}}_{p,q,m,m'}$, one has $v\in
\B^{\gamma_{1}-\gamma_{2}}_{\frac{m}{p},\infty}\subset\B^{\gamma_{1}-\gamma_{2}}_{0,\infty}$.
Hence
$$
|v^{0}_{j'-3,k''}(s)|\lesssim
s^{-\frac{\gamma_{2}-\gamma_{1}}{2\beta}}2^{-\frac{nj'}{2}},
$$
and consequently, by (i) of Lemma \ref{le6},
$$\begin{array}{rl}
|a^{\epsilon,1}_{j,k}(t)|
\lesssim&\!\!\!2^{j}\sum\limits_{|j-j'|\leq2}\sum\limits_{\epsilon',k'}\int^{2^{-1-{2j'\beta}}}_{0}
|u^{\epsilon'}_{j',k'}(s)|(1+|2^{j-j'}k'-k|)^{-N}e^{-\tilde{c}t2^{2j\beta}}s^{\frac{1}{2\beta}-1}ds.
\end{array}$$
Notice that $|j-j'|\leq2$. So, applying H\"older's inequality to $k'$
we get
$$\begin{array}{rl}
&I^{m,1}_{a, Q_{r}}(t)\\
&\quad\lesssim\ |Q_{r}|^{\frac{q\gamma_{2}}{n}-\frac{q}{p}}\sum\limits_{j\geq\max\{-\log_{2}r,
-\frac{\log_{2}t}{2\beta}\}}2^{qj(\gamma_{1}+\frac{n}{2}-\frac{n}{p})}\Big[\sum\limits_{(\epsilon,k)\in
S^{j}_{r}}\sum\limits_{|j-j'|\leq2}\sum\limits_{\epsilon',k'}2^{pj}e^{-\tilde{c}pt2^{2j\beta}}\\
&\quad(1+|2^{j-j'}k'-k|)^{-N}\Big(\int^{2^{-1-{2j'\beta}}}_{0}
|u^{\epsilon'}_{j',k'}(s)| s^{\frac{1}{2\beta}-1}ds\Big)^{p}
(t2^{2j\beta})^{m}\Big]^{\frac{q}{p}}.
\end{array}$$
By $p>2m'\beta$, $j\sim j'$ and H\"{o}lder's
inequality, we apply (\ref{eqn:est1}) to get
$$\begin{array}{rl}
I^{m,1}_{a, Q_{r}}(t)
&\lesssim\
|Q_{r}|^{\frac{q\gamma_{2}}{n}-\frac{q}{p}}\sum\limits_{j\geq\max\{-\log_{2}r,
-\frac{\log_{2}t}{2\beta}\}}2^{qj(\gamma_{1}+\frac{n}{2}-\frac{n}{p})}
\Big[\sum\limits_{(\epsilon,k)\in
S^{j}_{r}}e^{-cpt2^{2j\beta}}\\
&\sum\limits_{w\in\mathbb{Z}^{n}}(1+|w|)^{-N}
\sum\limits_{(\epsilon',k')\in
S^{w,j'}_{j,k}}\int^{2^{-1-2j'\beta}}_{0}|u^{\epsilon'}_{j',k'}(s)|^{p}(s2^{2j'\beta})^{m'}\frac{ds}{s}(t2^{2j\beta})^{m}\Big]^{\frac{q}{p}}.
\end{array}$$

If $q\leq p$, by the $\alpha$-triangle inequality we obtain
$$\begin{array}{rl}
I^{m,1}_{a, Q_{r}}(t)\lesssim&
|Q_{r}|^{\frac{q\gamma_{2}}{n}-\frac{q}{p}}\sum\limits_{j\geq\max\{-\log_{2}r,
-\frac{\log_{2}t}{2\beta}\}}2^{qj(\gamma_{1}+\frac{n}{2}-\frac{n}{p})}\sum\limits_{w\in\mathbb{Z}^{n}}(1+|w|)^{-\frac{qN}{p}}\\
&\Big[\sum\limits_{(\epsilon',k')\in
S^{w,j'}_{r}}\int^{2^{-1-2j'\beta}}_{0}|u^{\epsilon'}_{j',k'}(s)|^{p}(s2^{2j'\beta})^{m'}\frac{ds}{s}\Big]^{\frac{q}{p}}\\
\lesssim&\|u\|_{\B^{\gamma_{1},\gamma_{2}, IV}_{p,q,m'}}\lesssim1.
\end{array}$$

If $q>p$, H\"{o}lder's inequality implies that
$$\begin{array}{rl}
I^{m,1}_{a, Q_{r}}(t)\lesssim&\!\!\!
|Q_{r}|^{\frac{q\gamma_{2}}{n}-\frac{q}{p}}\sum\limits_{j\geq\max\{-\log_{2}r,
-\frac{\log_{2}t}{2\beta}\}}2^{qj(\gamma_{1}+\frac{n}{2}-\frac{n}{p})}\sum\limits_{w\in\mathbb{Z}^{n}}(1+|w|)^{-N}\\
&\Big[\sum\limits_{(\epsilon',k')\in
S^{w,j'}_{r}}\int^{2^{-1-2j'\beta}}_{0}|u^{\epsilon'}_{j',k'}(s)|^{p}(s2^{2j'\beta})^{m'}\frac{ds}{s}\Big]^{\frac{q}{p}}\\
\lesssim&\!\!\!\|u\|_{\B^{\gamma_{1},\gamma_{2},
IV}_{p,q,m'}}\lesssim1.
\end{array}$$

In a similar manner, we can obtain the following two assertions.
\begin{equation}\nonumber
\begin{cases}
\text{ Case 7.2: }
(t,x)\mapsto\sum\limits_{(\epsilon,j,k)}a^{\epsilon,2}_{j,k}(t)\Phi^{\epsilon}_{j,k}(x)\ \ \hbox{is\ in}\ \ \B^{\gamma_{1},\gamma_{2},I}_{p,q,m};\\
\text{ Case 7.3: } (t,x)\mapsto
\sum\limits_{(\epsilon,j,k)}a^{\epsilon,3}_{j,k}(t)\Phi^{\epsilon}_{j,k}(x)\ \ \hbox{is\ in}\ \ \B^{\gamma_{1},\gamma_{2},I}_{p,q,m}.
\end{cases}
\end{equation}

\subsection{The setting (ii)}
For $i=1, 2, 3$, define
$$\begin{array}{rl}
III^{m,i}_{a, Q_{r}}=&|Q_{r}|^{\frac{q\gamma_{2}}{n}-\frac{q}{p}}\sum\limits_{j\geq-\log_{2}r}2^{qj(\gamma_{1}+\frac{n}{2}-\frac{n}{p})}
\Big[\int^{r^{2\beta}}_{2^{-2j\beta}}\sum\limits_{(\epsilon,k)\in
S^{j}_{r}}|a^{\epsilon,i}_{j,k}(t)|^{p}(t2^{2j\beta})^{m}\frac{dt}{t}\Big]^{\frac{q}{p}}.
\end{array}$$
We consider the following three cases:
\begin{equation}\nonumber
\begin{cases}
\text{ Case 7.4: }
(t,x)\mapsto\sum\limits_{(\epsilon,j,k)\in\Lambda_{n}}a^{\epsilon,1}_{j,k}(t)\Phi^{\epsilon}_{j,k}(x)\ \ \hbox{is\ in}\ \ \B^{\gamma_{1},\gamma_{2},
III}_{p,q,m};\\
\text{ Case 7.5: } (t,x)\mapsto\sum\limits_{(\epsilon,j,k)\in\Lambda_{n}}a^{\epsilon,2}_{j,k}(t)\Phi^{\epsilon}_{j,k}(x)\ \ \hbox{is\ in}\ \ \B^{\gamma_{1},\gamma_{2},
III}_{p,q,m};\\
\text{Case 7.6: }(t,x)\mapsto\sum\limits_{(\epsilon,j,k)\in\Lambda_{n}}a^{\epsilon,3}_{j,k}(t)\Phi^{\epsilon}_{j,k}(x)\ \ \hbox{is\ in}\ \ \B^{\gamma_{1},\gamma_{2},III}_{p,q,m}.
\end{cases}
\end{equation}
It is enough to check Case 7.4 since Cases 7.5 and 7.6 can be dealt with similarly.
In fact, for $v\in\B^{\gamma_{1},\gamma_{2}}_{p,q,m,m'}$ we have, by (i) of Lemma \ref{le6},
$$\begin{array}{rl}
|a^{\epsilon,1}_{j,k}(t)|\lesssim&\!\!\!2^{j}\sum\limits_{|j-j'|\leq2}\sum\limits_{\epsilon',k'}\int^{2^{-1-2j'\beta}}_{0}
|u^{\epsilon'}_{j',k'}(s)|
e^{-\tilde{c}t2^{2j\beta}}(1+|2^{j-j'}k'-k|)^{-N}s^{\frac{1}{2\beta}}\frac{ds}{s},
\end{array}$$
whence obtaining
$$\begin{array}{rl}
III^{m,1}_{a, Q_{r}}\lesssim&|Q_{r}|^{\frac{q\gamma_{2}}{n}-\frac{q}{p}}\sum\limits_{j\geq-\log_{2}r}2^{qj(\gamma_{1}+\frac{n}{2}-\frac{n}{p})}
\Big[\int^{r^{2\beta}}_{2^{-2j\beta}}2^{pj}e^{-\tilde{c}pt2^{2j\beta}}
\sum\limits_{(\epsilon,k)\in
S^{j}_{r}}\\
&\Big(\sum\limits_{|j-j'|\leq2}\sum\limits_{\epsilon',k'}\int^{2^{-1-2j'\beta}}_{0}
|u^{\epsilon'}_{j',k'}(s)|
(1+|2^{j-j'}k'-k|)^{-N}s^{\frac{1}{2\beta}}\frac{ds}{s}\Big)^{p}(t2^{2j\beta})^{m}\frac{dt}{t}\Big]^{\frac{q}{p}}.
\end{array}$$
Applying H\"older's inequality to $k'$ and $s$ respectively, along with (\ref{eqn:est1}) and $|j-j'|\leq2$, we find
$$\begin{array}{rl}
III^{m,1}_{a, Q_{r}}\lesssim&|Q_{r}|^{\frac{q\gamma_{2}}{n}-\frac{q}{p}}\sum\limits_{j\geq-\log_{2}r}2^{qj(\gamma_{1}+\frac{n}{2}-\frac{q}{p})}
\Big[\sum\limits_{w\in\mathbb{Z}^{n}}(1+|w|)^{-N}\int^{r^{2\beta}}_{2^{-2j\beta}}2^{jp}e^{-\tilde{c}pt2^{2j\beta}}\\
&\sum\limits_{(\epsilon,k)\in
S^{j}_{r}}2^{-jp}\Big(\sum\limits_{(\epsilon',k')\in S^{w,j'}_{
j,k}}\int^{2^{-1-2j'\beta}}_{0}|u^{\epsilon'}_{j',k'}(s)|^{p}(s2^{2j\beta})^{m'}\frac{ds}{s}\Big)(t2^{2j\beta})^{m}\frac{dt}{t}\Big]^{\frac{q}{p}}.
\end{array}$$

If $q\leq p$, by changing variables we get
$$\begin{array}{rl}
III^{m,1}_{a, Q_{r}}\lesssim&
\sum\limits_{w\in\mathbb{Z}^{n}}(1+|w|)^{-\frac{qN}{p}}
|Q_{r}|^{\frac{q\gamma_{2}}{n}-\frac{q}{p}}\sum\limits_{j\geq-\log_{2}r}2^{qj(\gamma_{1}+\frac{n}{2}-\frac{q}{p})}\\
&\Big[\int^{r^{2\beta}}_{2^{-2j\beta}}e^{-\tilde{c}t2^{2j\beta}}(t2^{2j\beta})^{m}\Big(\sum\limits_{(\epsilon',k')\in S^{w,j'}_{
r}}\int^{2^{-1-2j'\beta}}_{0}|u^{\epsilon'}_{j',k'}(s)|^{p}(s2^{2j\beta})^{m'}\frac{ds}{s}\Big)\frac{dt}{t}\Big]^{\frac{q}{p}}\\
\lesssim&\|u\|_{\B^{\gamma_{1},\gamma_{2}, IV}_{p,q,m'}}.
\end{array}$$

If $q>p$, via applying H\"older's inequality for $w$ and
$\frac{q}{p}>1$, we similarly have
$$\begin{array}{rl}
III^{m,1}_{a, Q_{r}}\lesssim&\!\!\!\sum\limits_{w\in\mathbb{Z}^{n}}(1+|w|)^{-N}
|Q_{r}|^{\frac{q\gamma_{2}}{n}-\frac{q}{p}}\sum\limits_{j\geq-\log_{2}r}2^{qj(\gamma_{1}+\frac{n}{2}-\frac{q}{p})}\\
&\Big[\int^{r^{2\beta}}_{2^{-2j\beta}}e^{-\tilde{c}t2^{2j\beta}}(t2^{2j\beta})^{m}\Big(\sum\limits_{(\epsilon',k')\in S^{w,j'}_{
r}}\int^{2^{-1-2j'\beta}}_{0}|u^{\epsilon'}_{j',k'}(s)|^{p}(s2^{2j\beta})^{m'}\frac{ds}{s}\Big)\frac{dt}{t}\Big]^{\frac{q}{p}}\\
\lesssim&\!\!\!\|u\|_{\B^{\gamma_{1},\gamma_{2}, IV}_{p,q,m'}}.
\end{array}$$

\subsection{The setting (iii)}  The argument for that the function
$$\begin{array}{rl}
&(t,x)\mapsto\sum\limits_{(\epsilon,j,k)\in\Lambda_{n}}a^{\epsilon,4}_{j,k}(t)\Phi^{\epsilon}_{j,k}(x)\ \ \hbox{is\ in}\ \ \B^{\gamma_{1},\gamma_{2},II}_{p,q}
\bigcap \B^{\gamma_{1},\gamma_{2},IV}_{p,q, m'}\end{array}$$
is divided into two cases below.

 $$\begin{array}{rl}
 &\text{Case 7.7: }\
  (t,x)\mapsto\sum\limits_{(\epsilon,j,k)\in\Lambda_{n}}a^{\epsilon,4}_{j,k}(t)\Phi^{\epsilon}_{j,k}(x)\ \ \hbox{is\  in}\ \
\B^{\gamma_{1},\gamma_{2},II}_{p,q}.
\end{array}$$

Let
$$\begin{array}{rl}
II^{4}_{a, Q_{r}}(t)
&=\ |Q_{r}|^{\frac{q\gamma_{2}}{n}-\frac{q}{p}}\sum\limits_{-\log_{2}r\leq
j<-\frac{\log_{2}t}{2\beta}}2^{qj(\gamma_{1}+\frac{n}{2}-\frac{n}{p})}
\Big[\sum\limits_{(\epsilon,k)\in
S^{j}_{r}}|a^{\epsilon,4}_{j,k}(t)|^{p}\Big]^{\frac{q}{p}}.
\end{array}$$
Then, by (ii) of Lemma \ref{le6} we have

$$\begin{array}{rl}
|a^{\epsilon,4}_{j,k}(t)|
\lesssim&\!\!\!2^{j}\sum\limits_{|j-j'|\leq2}\sum\limits_{\epsilon',k'}\int^{t}_{0}
|u^{\epsilon'}_{j',k'}(s)| e^{-\tilde{c}(t-s)2^{2j\beta}}
(1+|2^{j-j'}k'-k|)^{-N}s^{\frac{1}{2\beta}-1}ds,
\end{array}$$
whence getting via H\"older's inequality and (\ref{eqn:est1})
$$\begin{array}{rl}
II^{4}_{a, Q_{r}}(t)
\lesssim&|Q_{r}|^{\frac{q\gamma_{2}}{n}-\frac{q}{p}}\sum\limits_{-\log_{2}r\leq
j<-\frac{\log_{2}t}{2\beta}}2^{qj(\gamma_{1}+\frac{n}{2}-\frac{n}{p})}
\Big[2^{pj}\sum\limits_{|j-j'|\leq2}\sum\limits_{w\in\mathbb{Z}^{n}}(1+|w|)^{-N}\\
& \sum\limits_{(\epsilon',k')\in S^{w,j'}_{r}}t^{p-1-\mu
p}\int^{t}_{0}|u^{\epsilon'}_{j',k'}(s)|^{p}s^{(\frac{1}{2\beta}-1+\mu)p}ds\Big]^{\frac{q}{p}}
\end{array}$$

If $q>p$, by H\"older's inequality we have
$$\begin{array}{rl}
II^{4}_{a, Q_{r}}(t)\ \lesssim&|Q_{r}|^{\frac{q\gamma_{2}}{n}-\frac{q}{p}}\sum\limits_{-\log_{2}r\leq
j<-\frac{\log_{2}t}{2\beta}}2^{qj(\gamma_{1}+\frac{n}{2}-\frac{n}{p})}\sum\limits_{w\in\mathbb{Z}^{n}}(1+|w|)^{-N}\\
&2^{qj}\sum\limits_{|j-j'|\leq2}t^{\frac{(p-1-p\mu)q}{p}}\Big[\int^{t}_{0}\sum\limits_{(\epsilon',k')\in
S^{w,j'}_{r}}|u^{\epsilon'}_{j',k'}(s)|^{p}s^{(\frac{1}{2\beta}-1+\mu)p}ds\Big]^{\frac{q}{p}}.
\end{array}$$
Because $2^{2j\beta}t\leq1$, one has $2^{qj}\leq t^{-\frac{q}{2\beta}}$.
This in turn gives
$$\begin{array}{rl}
II^{4}_{a, Q_{r}}(t)\ \lesssim&|Q_{r}|^{\frac{q\gamma_{2}}{n}-\frac{q}{p}}\sum\limits_{-\log_{2}r\leq
j<-\frac{\log_{2}t}{2\beta}}2^{qj(\gamma_{1}+\frac{n}{2}-\frac{n}{p})}\sum\limits_{w\in\mathbb{Z}^{n}}(1+|w|)^{-N}2^{qj}\\
&\sum\limits_{|j-j'|\leq2}t^{\frac{q}{2\beta}}t^{-p(\frac{1}{2\beta}-1+\mu)-1}
\Big[\int^{t}_{0}(\sum\limits_{\epsilon',k'}|u^{\epsilon'}_{j',k'}(s)|^{p})^{\frac{q}{p}}
s^{(\frac{1}{2\beta}-1+\mu)p}ds\Big]\\
\lesssim&\|u\|_{\B^{\gamma_{1},\gamma_{2},II}_{p,q}}.
\end{array}$$

If $q\leq p$, we have
$$\begin{array}{rl}
II^{4}_{a, Q_{r}}(t)\ \lesssim&|Q_{r}|^{\frac{q\gamma_{2}}{n}-\frac{q}{p}}\sum\limits_{-\log_{2}r\leq
j<-\frac{\log_{2}t}{2\beta}}2^{qj(\gamma_{1}+\frac{n}{2}-\frac{n}{p})}\\
&\Big[\sum\limits_{w\in\mathbb{Z}^{n}}(1+|w|)^{-N}\sum\limits_{|j-j'|\leq2}2^{jp}
\sum\limits_{(\epsilon',k')\in
S^{w,j'}_{r}}\Big(\int^{t}_{0}|u^{\epsilon'}_{j',k'}(s)|s^{\frac{1}{2\beta}-1}ds\Big)^{p}\Big]^{\frac{q}{p}}.
\end{array}$$
Because $t\leq2^{-2j\beta}$ and $0<m'<\min\{1, \frac{p}{2\beta}\}$,
H\"older's inequality implies
$$\begin{array}{rl}
II^{4}_{a, Q_{r}}(t)
&\lesssim|Q_{r}|^{\frac{q\gamma_{2}}{n}-\frac{q}{p}}\sum\limits_{j\geq-\log_{2}r}2^{qj(\gamma_{1}+\frac{n}{2}-\frac{n}{p})}
\sum\limits_{w\in\mathbb{Z}^{n}}(1+|w|)^{-\frac{qN}{p}}\\
&\quad \sum\limits_{|j-j'|\leq2}\Big[2^{pj}\int^{2^{-2j\beta}}_{0}
\sum\limits_{(\epsilon',k')\in S^{w,j'}_{r}}|u^{\epsilon'}_{j',k'}(s)|^{p}s^{m'}2^{-2j'\beta(\frac{p}{2\beta}-m')}\frac{ds}{s}\Big]^{\frac{q}{p}}\\
&\lesssim\|u\|_{\B^{\gamma_{1},\gamma_{2},IV}_{p,q,m'}}.
\end{array}$$

$$\begin{array}{rl}
&\text{Case 7.8:}\ \ (t,x)\mapsto\sum\limits_{(\epsilon,j,k)\in\Lambda_{n}}a^{\epsilon,4}_{j,k}(t)\Phi^{\epsilon}_{j,k}(x)\ \ \hbox{is\ in}\ \
\B^{\gamma_{1},\gamma_{2},IV}_{p,q, m'}.
\end{array}$$

Similarly, let
$$\begin{array}{rl}
IV^{4,m'}_{a,Q_{r}}=&\!\!\!|Q_{r}|^{\frac{q\gamma_{2}}{n}-\frac{q}{p}}\sum\limits_{j\geq-\log_{2}r}2^{qj(\gamma_{1}+\frac{n}{2}-\frac{n}{p})}
\Big[\int^{2^{-2j\beta}}_{0}\sum\limits_{(\epsilon,k)\in
S^{j}_{r}}|a^{\epsilon,4}_{j,k}(t)|^{p}(t2^{2j\beta})^{m'}\frac{dt}{t}\Big]^{\frac{q}{p}}.
\end{array}$$
Choosing a constant $\mu$ such that $m'+p-1-\frac{p}{2\beta}\leq
p\mu<p-1$, we use (\ref{eqn:est1}) and H\"older's inequality to get
$$\begin{array}{rl}
IV^{4,m'}_{a,Q_{r}}
\lesssim&|Q_{r}|^{\frac{q\gamma_{2}}{n}-\frac{q}{p}}\sum\limits_{j\geq-\log_{2}r}2^{qj(\gamma_{1}+\frac{n}{2}-\frac{n}{p})}
\Big[\int^{2^{-2j\beta}}_{0}2^{pj}\sum\limits_{|j-j'|\leq2}\sum\limits_{w\in\mathbb{Z}^{n}}(1+|w|)^{-N}\\
&\sum\limits_{(\epsilon',k')\in S^{w,j'}_{r}}t^{p-1-\mu
p}\Big(\int^{t}_{0}|u^{\epsilon'}_{j',k'}(s)|^{p}s^{(\frac{1}{2\beta}-1+\mu)p}ds\Big)(t2^{2j\beta})^{m'}\frac{dt}{t}\Big]^{\frac{q}{p}}.
\end{array}$$

If $q\leq p$, by the $\alpha-$triangle inequality we have
$$\begin{array}{rl}
IV^{4,m'}_{a,Q_{r}}
\lesssim&\ \sum\limits_{w\in\mathbb{Z}^{n}}(1+|w|)^{-\frac{qN}{p}}|Q_{r}|^{\frac{q\gamma_{2}}{n}-\frac{q}{p}}
\sum\limits_{j\geq-\log_{2}r}2^{qj(\gamma_{1}+\frac{n}{2}-\frac{n}{p})}\sum\limits_{|j-j'|\leq2}2^{qj} \\
&\Big[\int^{2^{-2j\beta}}_{0}\sum\limits_{(\epsilon',k')\in
S^{w,j'}_{r}}|u^{\epsilon'}_{j',k'}(s)|^{p}s^{\frac{p}{2\beta}-p+p\mu}\Big(\int^{2^{-2j\beta}}_{s}2^{2mj'\beta}
t^{m'+p-\mu p}\frac{dt}{t^{2}}\Big)ds\Big]^{\frac{q}{p}}.
\end{array}$$

 Because $s2^{2j'\beta}\leq1$ and $m'+p-1-\frac{p}{2\beta}\leq
 p\mu$,  we obtain
$$\begin{array}{rl}
IV^{4,m'}_{a,Q_{r}}\lesssim&\!\!\!
\sum\limits_{w\in\mathbb{Z}^{n}}(1+|w|)^{-\frac{qN}{p}}|Q_{r}|^{\frac{q\gamma_{2}}{n}-\frac{q}{p}}
\sum\limits_{j'\geq-\log_{2}r_{w}}2^{qj'(\gamma_{1}+\frac{n}{2}-\frac{n}{p})}\\
&\Big[\int^{2^{-2(j'-2)\beta}}_{0}\sum\limits_{(\epsilon',k')\in
S^{w,j'}_{r}}|u^{\epsilon'}_{j',k'}(s)|^{p}(s2^{2j'\beta})^{\frac{p}{2\beta}+1-p+p\mu}\frac{ds}{s}\Big]^{\frac{q}{p}}\\
\lesssim&\!\!\!\|u\|_{\B^{\gamma_{1},\gamma_{2},III}_{p,q,m}}+\|u\|_{\B^{\gamma_{1},\gamma_{2},IV}_{p,q,m'}}.
\end{array}$$

If $q>p$, by H\"older's inequality,  we have
$$\begin{array}{rl}
IV^{4,m'}_{a,Q_{r}}\lesssim&\!\!\!
\sum\limits_{w\in\mathbb{Z}^{n}}|Q_{r}|^{\frac{q\gamma_{2}}{n}-\frac{q}{p}}(1+|w|)^{-N}
\sum\limits_{j\geq-\log_{2}r_{w}}2^{qj'(\gamma_{1}+\frac{n}{2}-\frac{n}{p})}\\
&\Big[\int^{2^{-2(j'-2)\beta}}_{0}\sum\limits_{(\epsilon',k')\in
S^{w,j'}_{r}}|u^{\epsilon'}_{j',k'}(s)|^{p}(s2^{2j'\beta})^{\frac{p}{2\beta}+1-p+p\mu}\frac{ds}{s}\Big]^{\frac{q}{p}}\\
\lesssim&\!\!\!\|u\|_{\B^{\gamma_{1},\gamma_{2},III}_{p,q,m}}+\|u\|_{\B^{\gamma_{1},\gamma_{2},IV}_{p,q,m'}}.
\end{array}$$

\section{Proof of Lemma \ref{lem54}}\label{sec8}
\subsection{The setting (i)}
 For $i=1,2,3$, define
 $$\begin{array}{rl}
 I^{m,i}_{b, Q_{r}}(t)=&\!\!\!|Q_{r}|^{\frac{q\gamma_{2}}{n}-\frac{q}{p}}\sum\limits_{j\geq\max\{-\log_{2}r,-\frac{\log_{2}t}{2\beta}\}}
2^{qj(\gamma_{1}+\frac{n}{2}-\frac{n}{p})}\Big[\sum\limits_{(\epsilon,k)\in S^{j}_{r}}|b^{\epsilon,i}_{j,k}(t)|^{p}(t2^{2j\beta})^{m}\Big]^{\frac{q}{p}}.
 \end{array}$$
 We divide the proof into three cases:
 \begin{equation}\nonumber
\begin{cases}
\text{ Case 8.1: }
(t,x)\mapsto\sum\limits_{(\epsilon,j,k)\in
\Lambda_{n}}b^{\epsilon,1}_{j,k}(t)\Phi^{\epsilon}_{j,k}(x)\ \ \hbox{is\ in}\ \ \B^{\gamma_{1},\gamma_{2},I}_{p,q,m};\\
\text{ Case 8.2: } (t,x)\mapsto\sum\limits_{(\epsilon,j,k)\in
\Lambda_{n}}b^{\epsilon,2}_{j,k}(t)\Phi^{\epsilon}_{j,k}(x)\ \ \hbox{is\ in}\ \ \B^{\gamma_{1},\gamma_{2},I}_{p,q,m};\\
\text{Case 8.3: }(t,x)\mapsto\sum\limits_{(\epsilon,j,k)\in
\Lambda_{n}}b^{\epsilon,3}_{j,k}(t)\Phi^{\epsilon}_{j,k}(x)\ \ \hbox{is\ in}\ \ \B^{\gamma_{1},\gamma_{2},I}_{p,q,m}.
\end{cases}
\end{equation}
But, we only demonstrate Case 8.1 and omit the proofs of Cases 8.2 and 8.3 due to their similarity.

Assume first  $1<p\le 2$. If $s2^{2j'\beta}\leq 1$, then
$$
v\in\B^{\gamma_{1},\gamma_{2}}_{p,q,m,m'}\subset
\B^{\gamma_{1}-\gamma_{2}}_{\frac{m}{p},\infty}\Rightarrow |v^{\epsilon''}_{j',k''}(s)|\lesssim
2^{-(\frac{n}{2}+\gamma_{1}-\gamma_{2})j'}.
$$
Because  $v\in\B^{\gamma_{1},\gamma_{2},II}_{p,q}$, we have
$$\begin{array}{rl}
\sum\limits_{(\epsilon'',k'')\in S^{w',j'}_{
j,k}}|v^{\epsilon''}_{j',k''}(s)|^{p}\lesssim
2^{-nj+p\gamma_{2}j}2^{-pj'(\gamma_{1}+\frac{n}{2}-\frac{n}{p})}.
\end{array}$$
By Lemma \ref{le7} (i), (\ref{eqn:est3}) and H\"older's inequality we get
$$\begin{array}{rl}
&|b^{\epsilon,1}_{j,k}(t)|\\
&\quad\lesssim\ 2^{\frac{nj}{2}+j}e^{-ct2^{2j\beta}}2^{j(-n+p\gamma_{2})(1-\frac{1}{p})}\sum\limits_{w\in\mathbb{Z}^{n}}(1+|w|)^{-N}\sum\limits_{j\leq
j'+2}2^{-(2-p)(\frac{n}{2}+\gamma_{1}-\gamma_{2})j'}\\
&\quad 2^{-j'(\gamma_{1}+\frac{n}{2}-\frac{n}{p})(p-1)}2^{-2j'\beta(1-\frac{1}{p})}
\Big(\int^{2^{-1-2j'\beta}}_{0}\sum\limits_{(\epsilon',k')\in S^{w,j'}_{
j,k}}|u^{\epsilon'}_{j',k'}(s)|^{p}ds\Big)^{\frac{1}{p}}.
\end{array}$$
This in turn yields
$$\begin{array}{rl}
 I^{m,1}_{b, Q_{r}}(t)\lesssim&\!\!\!|Q_{r}|^{\frac{q\gamma_{2}}{n}-\frac{q}{p}}\sum\limits_{j\geq\max\{-\log_{2}r,-\frac{\log_{2}t}{2\beta}\}}
2^{qj(\gamma_{1}+\frac{n}{2}-\frac{q}{p})}2^{\frac{qnj}{2}+qj}2^{qj(-n+p\gamma_{2})(1-\frac{1}{p})}\\
&\Big\{\sum\limits_{w\in\mathbb{Z}^{n}}(1+|w|)^{-N}\sum\limits_{(\epsilon,k)\in
S^{j}_{r}}\Big[
\sum\limits_{j<j'+2}2^{-(2-p)(\frac{n}{2}+\gamma_{1}-\gamma_{2})j'}
2^{-j'(\gamma_{1}+\frac{n}{2}-\frac{n}{p})(p-1)}\\
&2^{-2j'\beta(1-\frac{1}{p})}\Big(\int^{2^{-1-2j'\beta}}_{0}\sum\limits_{(\epsilon',k')\in S^{w,j'}_{
j,k}}
|u^{\epsilon'}_{j',k'}(s)|^{p}ds\Big)^{\frac{1}{p}}\Big]^{p}\Big\}^{\frac{q}{p}}.
\end{array}$$
Since $0<s<2^{-1-2j'\beta}$ and $m'<1$, one has
$(s2^{2j'\beta})^{1-m'}\lesssim1$. This implies
$$\begin{array}{rl}
 I^{m,1}_{b, Q_{r}}(t)
&\lesssim|Q_{r}|^{\frac{q\gamma_{2}}{n}-\frac{q}{p}}\sum\limits_{j\geq\max\{-\log_{2}r,-\frac{\log_{2}t}{2\beta}\}}
2^{qj(\gamma_{1}+\frac{n}{2}-\frac{q}{p})}2^{\frac{qnj}{2}+qj}2^{qj(-n+p\gamma_{2})(1-\frac{1}{p})}\\
&\quad\Big\{\sum\limits_{w\in\mathbb{Z}^{n}}\frac{1}{(1+|w|)^{N}}\sum\limits_{(\epsilon,k)\in
S^{j}_{r}}\sum\limits_{j<j'+2}2^{\delta(j'-j)}2^{-p(2-p)(\frac{n}{2}+\gamma_{1}-\gamma_{2})j'}
2^{-pj'(\gamma_{1}+\frac{n}{2}-\frac{n}{p})(p-1)}\\
&\quad\quad2^{-2pj'\beta}\Big(\int^{2^{-1-2j'\beta}}_{0}\sum\limits_{(\epsilon',k')\in S^{w,j'}_{
j,k}}|u^{\epsilon'}_{j',k'}(s)|^{p}(s2^{2j'\beta})^{m'}\frac{ds}{s}\Big)\Big\}^{\frac{q}{p}}.\\
\end{array}$$

If $q\leq p$, for $p\gamma_{2}+2-2\beta>0$ take
$0<\delta<p(p\gamma_{2}+2-2\beta)$. By the $\alpha-$triangle inequality, we
get
$$\begin{array}{rl}
 I^{m,1}_{b, Q_{r}}(t)\lesssim&\!\!\!|Q|^{\frac{q\gamma_{2}}{n}-\frac{q}{p}}\sum\limits_{w\in\mathbb{Z}^{n}}(1+|w|)^{-\frac{qN}{p}}
\sum\limits_{j\geq\{-\log_{2}r,-\frac{\log_{2}t}{2\beta}\}}\sum\limits_{j\leq
j'+2}2^{q(j'-j)[\frac{\delta}{p}-(p\gamma_{2}+2-2\beta)]}\\
&2^{qj'(\gamma_{1}+\frac{n}{2}-\frac{n}{p})}\Big(\int^{2^{-1-2j'\beta}}_{0}\sum\limits_{(\epsilon',k')\in S^{w,j'}_{
r}}|u^{\epsilon'}_{j',k'}(s)|^{p}(s2^{2j'\beta})^{m'}\frac{ds}{s}\Big)^{\frac{q}{p}}
\lesssim\|u\|_{\B^{\gamma_{1},\gamma_{2}, IV}_{p,q,m}}.
\end{array}$$

If $q>p$, by H\"older's inequality we get
$$\begin{array}{rl}
I^{m,1}_{b, Q_{r}}(t)
&\lesssim\sum\limits_{w\in\mathbb{Z}^{n}}\frac{|Q_{r}|^{\frac{q\gamma_{2}}{n}-\frac{q}{p}}}{(1+|w|)^{N}}
\sum\limits_{j\geq\max\{-\log_{2}r,-\frac{\log_{2}t}{2\beta}\}}\Big[\sum\limits_{j\leq
j'+2}2^{(j'-j)[\delta-(p\gamma_{2}+2-2\beta)]}\Big]^{\frac{q-p}{p}}\\
&\quad \times\Big\{\sum\limits_{j\leq
j'+2}2^{(j'-j)[\delta-(p\gamma_{2}+2-2\beta)]}2^{qj'(\gamma_{1}+\frac{n}{2}-\frac{n}{p})}\\
&\quad\quad\Big(\int^{2^{-1-2j'\beta}}_{0}\sum\limits_{(\epsilon',k')\in S^{w,j'}_{
r}}|u^{\epsilon'}_{j',k'}(s)|^{p}(s2^{2j'\beta})^{m'}\frac{ds}{s}\Big)^{\frac{q}{p}}\Big\}
\lesssim\|u\|_{\B^{\gamma_{1},\gamma_{2}, IV}_{p,q,m'}}.
\end{array}$$

Assume then $2<p<\infty$. Because $0<s<2^{-1-2j'\beta}$,
$v\in\B^{\gamma_{1},\gamma_{2},II}_{p,q}$ implies
$$\begin{array}{rl}
&\Big(\int^{2^{-1-2j'\beta}}_{0}\sum\limits_{(\epsilon'',k'')\in S^{w',j'}_{j,k}}|v^{\epsilon''}_{j',k''}(s)|^{p}ds\Big)^{\frac{1}{p}}\lesssim
2^{\gamma_{2}j-\frac{nj}{p}-j'(\gamma_{1}+\frac{n}{2}-\frac{n}{p})}2^{-\frac{2j'\beta}{p}}.
\end{array}$$
In view of Lemma \ref{le7} (i), H\"older's inequality and (\ref{eqn:est6}) we achieve
$$\begin{array}{rl}
|b^{\epsilon,1}_{j,k}(t)|
\lesssim&2^{\frac{nj}{2}+j}e^{-ct2^{2j\beta}}\sum\limits_{w\in\mathbb{Z}^{n}}(1+|w|)^{-N}\sum\limits_{j<j'+2}
2^{n(j'-j)(1-\frac{2}{p})}2^{-2j'\beta(1-\frac{2}{p})}\\
&2^{\gamma_{2}j-\frac{nj}{p}}2^{-j'(\gamma_{1}+\frac{n}{2}-\frac{n}{p})}2^{-\frac{2j'\beta}{p}}
\Big(\int^{2^{-1-2j'\beta}}_{0}\sum\limits_{(\epsilon',k')\in S^{w,j'}_{j,k}}|u^{\epsilon'}_{j',k'}(s)|^{p}ds\Big)^{\frac{1}{p}}.
\end{array}$$
Using the above estimate we obtain
$$\begin{array}{rl}
I^{m,1}_{b, Q_{r}}(t)
\lesssim&\ |Q_{r}|^{\frac{q\gamma_{2}}{n}-\frac{q}{p}}\sum\limits_{j\geq\max\{-\log_{2}r,-\frac{\log_{2}t}{2\beta}\}}
2^{qj(\gamma_{1}+\frac{n}{2}-\frac{n}{p})}2^{\frac{qnj}{2}+qj}2^{qj(\gamma_{2}-\frac{n}{p})}\\
&\Big\{\sum\limits_{(\epsilon,k)\in
S^{j}_{r}}e^{-cpt2^{2j\beta}}(t2^{2j\beta})^{m}\sum\limits_{w\in\mathbb{Z}^{n}}\frac{1}{(1+|w|)^{N}}\Big[
\sum\limits_{j<j'+2}
2^{n(j'-j)(1-\frac{2}{p})}2^{-j'\beta(2-\frac{4}{p})}\\
&2^{-j'(\gamma_{1}+\frac{n}{2}-\frac{n}{p})} 2^{-\frac{2j'\beta}{p}}
\Big(\int^{2^{-1-2j'\beta}}_{0}\sum\limits_{(\epsilon',k')\in S^{w,j'}_{j,k}}|u^{\epsilon'}_{j',k'}(s)|^{p}ds\Big)^{\frac{1}{p}}\Big]^{p}\Big\}^{\frac{q}{p}}.
\end{array}$$
Notice that $0<s<2^{-1-2j'\beta}$ and $m'<1$. For any $0<\delta<(\gamma_{1}+\gamma_{2}+1)$ we have
$$\begin{array}{rl}
I^{m,1}_{b, Q_{r}}(t)\lesssim&|Q_{r}|^{\frac{q\gamma_{2}}{n}-\frac{q}{p}}\sum\limits_{j\geq\max\{-\log_{2}r,-\frac{\log_{2}t}{2\beta}\}}
2^{qj(\gamma_{1}+\gamma_{2}+1+n-\frac{2n}{p})}\\
&\Big\{\sum\limits_{w\in\mathbb{Z}^{n}}(1+|w|)^{-N}\sum\limits_{j<j'+2}2^{(\delta+pn-2n)(j'-j)}
2^{-2j'\beta(p-2)}2^{-pj'(\gamma_{1}+\frac{n}{2}-\frac{n}{p})}\\
&2^{-4j'\beta}\Big[\int^{2^{-1-2j'\beta}}_{0}\sum\limits_{(\epsilon',k')\in S^{w,j'}_{r}}|u^{\epsilon'}_{j',k'}(s)|^{p}(s2^{2j'\beta})^{m'}\frac{ds}{s}\Big]\Big\}^{\frac{q}{p}}.
\end{array}$$

If $q\leq p$, the $\alpha-$triangle inequality implies
$$\begin{array}{rl}
I^{m,1}_{b, Q_{r}}(t)\lesssim&\!\!\!|Q_{r}|^{\frac{q\gamma_{2}}{n}-\frac{q}{p}}\sum\limits_{w\in\mathbb{Z}^{n}}(1+|w|)^{-\frac{qN}{p}}
\sum\limits_{j\geq\max\{-\log_{2}r,-\frac{\log_{2}t}{2\beta}\}}\sum\limits_{j<j'+2}2^{q(\frac{\delta}{p}-\gamma_{1}-\gamma_{2}-1)(j'-j)}\\
&2^{qj'(\gamma_{1}+\frac{n}{2}-\frac{n}{p})}\Big[\int^{2^{-1-2j'\beta}}_{0}\sum\limits_{(\epsilon',k')\in S^{w,j'}_{r}}|u^{\epsilon'}_{j',k'}(s)|^{p}(s2^{2j'\beta})^{m'}\frac{ds}{s}\Big]^{\frac{q}{p}}
\lesssim\|u\|_{\B^{\gamma_{1},\gamma_{2},IV}_{p,q,m'}}.
\end{array}$$

If $q>p$, by H\"older's inequality we obtain
$$\begin{array}{rl}
I^{m,1}_{b, Q_{r}}(t)
&\lesssim\ \sum\limits_{w\in\mathbb{Z}^{n}}\frac{|Q_{r}|^{\frac{q\gamma_{2}}{n}-\frac{q}{p}}}{(1+|w|)^{N}}
\sum\limits_{j\geq\max\{-\log_{2}r,-\frac{\log_{2}t}{2\beta}\}}
\Big(\sum\limits_{j<j'+2}2^{p(j'-j)(\frac{\delta}{p}-\gamma_{1}-\gamma_{2}-1)}\Big)^{\frac{q-p}{p}}\\
&\quad\Big\{\sum\limits_{j<j'+2}2^{(j'-j)[\delta-p(\gamma_{1}+\gamma_{2}+1)]}2^{qj'(\gamma_{1}+\frac{n}{2}-\frac{n}{p})}\\
&\quad\quad\Big[\int^{2^{-1-2j'\beta}}_{0}\sum\limits_{(\epsilon',k')\in S^{w,j'}_{r}}|u^{\epsilon'}_{j',k'}(s)|^{p}(s2^{2j'\beta})^{m'}\frac{ds}{s}\Big]^{\frac{q}{p}}\Big\}\lesssim
 \|u\|_{\B^{\gamma_{1},\gamma_{2},IV}_{p,q,m'}}.
\end{array}$$

\subsection{The setting (ii)}\label{sec9}
For $i=1,2,3$, define
$$\begin{array}{rl}
III^{m,i}_{b,Q_{r}}=&\!\!\!|Q_{r}|^{\frac{q\gamma_{2}}{n}-\frac{q}{p}}\sum\limits_{j\geq-\log_{2}r}2^{qj(\gamma_{1}+\frac{n}{2}-\frac{n}{p})}
\Big[\int^{r^{2\beta}}_{2^{-2j\beta}}\sum\limits_{(\epsilon,k)\in
S^{j}_{r}}|b^{\epsilon,i}_{j,k}(t)|^{p}(t2^{2j\beta})^{m}\frac{dt}{t}\Big]^{\frac{q}{p}}.
\end{array}$$
In a way similar to the setting (i) of the subsection 8.1, we may only handle the situation $1<p\leq2$. Under this the argument is split into three cases.
 \begin{equation}\nonumber
\begin{cases}
\text{ Case 8.4: }
(t,x)\mapsto\sum\limits_{(\epsilon,j,k)\in\Lambda_{n}}b^{\epsilon,1}_{j,k}(t)\Phi^{\epsilon}_{j,k}(x)\ \ \hbox{is\ in}\ \
\B^{\gamma_{1},\gamma_{2}, III}_{p,q,m};\\
\text{ Case 8.5: } (t,x)\mapsto\sum\limits_{(\epsilon,j,k)\in\Lambda_{n}}b^{\epsilon,2}_{j,k}(t)\Phi^{\epsilon}_{j,k}(x)\ \ \hbox{is\ in}\ \
\B^{\gamma_{1},\gamma_{2}, III}_{p,q,m};\\
\text{Case 8.6: }(t,x)\mapsto\sum\limits_{(\epsilon,j,k)\in\Lambda_{n}}b^{\epsilon,3}_{j,k}(t)\Phi^{\epsilon}_{j,k}(x)\ \ \hbox{is\ in}\ \
\B^{\gamma_{1},\gamma_{2}, III}_{p,q,m}.
\end{cases}
\end{equation}
It suffices to treat Case 8.4 since Cases 8.5 and 8.6 can be verified similarly.

Because $u$ and $v$ both belong to
$\B^{\gamma_{1}-\gamma_{2}}_{\frac{m}{p},\infty}$, for
$s2^{2j'\beta}\leq1$ we have
$|v^{\epsilon''}_{j',k''}(s)|\lesssim2^{-(\frac{n}{2}+\gamma_{1}-\gamma_{2})j'}.$
Owing to $v\in\B^{\gamma_{1},\gamma_{2},II}_{p,q}$ we also have
$$\begin{array}{rl}
&\sum\limits_{(\epsilon'',k'')\in S^{w',j'}_{j,k}}|v^{\epsilon''}_{j',k''}(s)|^{p}\lesssim2^{p\gamma_{2}j-nj}2^{-pj'(\gamma_{1}+\frac{n}{2}-\frac{n}{p})}\ \ \hbox{for\ a\ fixed}\ \ j'.\end{array}$$
This, along with H\"older's inequality, (\ref{eqn:est3}) and Lemma \ref{le7} (i), derives
$$\begin{array}{rl}
|b^{\epsilon,1}_{j,k}(t)|
&\lesssim\ 2^{[\frac{n}{p}-\frac{n}{2}+(p-1)\gamma_{2}+1]j}e^{-ct2^{2j\beta}}\sum\limits_{j\leq
j'+2}\sum\limits_{w\in\mathbb{Z}^{n}}(1+|w|)^{-N}\\
&\quad 2^{[\frac{n}{2}-\frac{n}{p}+\frac{2\beta}{p}-(p-1)\gamma_{2}-1]j'}\Big(\int^{2^{-1-2j'\beta}}_{0}\sum\limits_{(\epsilon',k')\in S^{w,j'}_{j,k}}|u^{\epsilon'}_{j',k'}(s)|^{p}ds\Big)^{\frac{1}{p}}.
\end{array}$$
The last estimate for $|b^{\epsilon,1}_{j,k}(t)|$ and (\ref{eqn:est5}) are used to derive
$$\begin{array}{rl}
III^{m,1}_{b,Q_{r}}
\lesssim&|Q_{r}|^{\frac{q\gamma_{2}}{n}-\frac{q}{p}}\sum\limits_{j\geq-\log_{2}r}2^{qj(p\gamma_{2}+2-2\beta)}\Big\{\sum\limits_{j\leq
j'+2}2^{\delta(j'-j)}2^{-pj'(p\gamma_{2}+2-2\beta)}\\
&2^{pj'(\gamma_{1}+\frac{n}{2}-\frac{n}{p})}\int^{2^{-1-2j'\beta}}_{0}\sum\limits_{(\epsilon',k')\in S^{w,j'}_{r}}|u^{\epsilon'}_{j',k'}(s)|^{p}(s2^{2j'\beta})^{m'}\frac{ds}{s}\Big\}^{\frac{q}{p}}.
\end{array}$$

If $q\leq p$, by the $\alpha-$triangle inequality we have
$$\begin{array}{rl}
III^{m,1}_{b,Q_{r}}\lesssim&\!\!\!|Q_{r}|^{\frac{q\gamma_{2}}{n}-\frac{q}{p}}\sum\limits_{j\geq-\log_{2}r}\sum\limits_{j\leq
j'+2}2^{q(j-j')(p\gamma_{2}+2-2\beta-\frac{\delta}{p})}2^{qj'(\gamma_{1}+\frac{n}{2}-\frac{n}{p})}\\
&\Big[\int^{2^{-1-2j'\beta}}_{0}\sum\limits_{(\epsilon',k')\in S^{w,j'}_{r}}|u^{\epsilon'}_{j',k'}(s)|^{p}(s2^{2j'\beta})^{m'}\frac{ds}{s}\Big]^{\frac{q}{p}}.
\end{array}$$
Changing the order of $j$ and $j'$, we find
$I^{5}_{Q_{r}}\lesssim\|u\|_{\B^{\gamma_{1},\gamma_{2},
IV}_{p,q,m'}}$.

If $q>p$, by H\"older's inequality we get
$$\begin{array}{rl}
III^{m,1}_{b,Q_{r}}
\lesssim&|Q_{r}|^{\frac{q\gamma_{2}}{n}-\frac{q}{p}}\sum\limits_{j\geq-\log_{2}r}
\Big(\sum\limits_{j\leq
j'+2}2^{p(j-j')(p\gamma_{2}+2-2\beta-\delta)}\Big)^{\frac{q-p}{p}}\\
&\Big\{\sum\limits_{j\leq
j'+2}2^{q(j-j')(p\gamma_{2}+2-2\beta-\delta)}\Big[\int^{2^{-1-2j'\beta}}_{0}\sum\limits_{(\epsilon',k')\in S^{w,j'}_{r}}|u^{\epsilon'}_{j',k'}(s)|^{p}(s2^{2j'\beta})^{m'}\frac{ds}{s}\Big]^{\frac{q}{p}}\Big\}.
\end{array}$$
Upon taking $0<\delta<p\gamma_{2}+2-2\beta$ and changing the order of $j$ and $j'$, we reach
$III^{m,1}_{b,Q_{r}}\lesssim\|u\|_{\B^{\gamma_{1},\gamma_{2},
IV}_{p,q,m'}}$.

\subsection{The setting (iii)} Like
the subsection 8.2, it is sufficient to deal with the situation $1<p\leq2$ below. For $i=4,5$, define
$$\begin{array}{rl}
II^{i}_{b,Q_{r}}(t)=&\!\!\!|Q_{r}|^{\frac{q\gamma_{2}}{n}-\frac{q}{p}}\sum\limits_{-\log_{2}r<j<-\frac{\log_{2}t}{2\beta}}2^{qj(\gamma_{1}+\frac{n}{2}-\frac{n}{p})}
\Big(\sum\limits_{(\epsilon,k)\in S^{j}_{r}}|b^{\epsilon,i}_{j,k}(t)|^{p}\Big)^{\frac{q}{p}}.
\end{array}$$
 We divide the argument into two cases.
 \begin{equation}\nonumber
\begin{cases}
\text{ Case 8.7: }
(t,x)\mapsto\sum\limits_{(\epsilon,j,k)\in\Lambda_{n}}b^{\epsilon,4}_{j,k}(t)\Phi^{\epsilon}_{j,k}(x)\ \ \hbox{is\ in}\ \
\B^{\gamma_{1},\gamma_{2}, II}_{p,q};\\
\text{ Case 8.8: } (t,x)\mapsto\sum\limits_{(\epsilon,j,k)\in\Lambda_{n}}b^{\epsilon,5}_{j,k}(t)\Phi^{\epsilon}_{j,k}(x)\ \ \hbox{is\ in}\ \
\B^{\gamma_{1},\gamma_{2}, II}_{p,q}.\\
\end{cases}
\end{equation}
In view of the settings (i) \& (ii) above, we only give a proof of Case 8.7 since the proof of Case 8.8 is similar. Due to $v\in\B^{\gamma_{1},\gamma_{2}}_{p,q,m,m'}$, one gets
$|v^{\epsilon''}_{j',k''}(s)|\lesssim2^{-(\frac{n}{2}+\gamma_{1}-\gamma_{2})j'}.$ For $0<s<2^{-2j'\beta}$, one has
$$
v\in\B^{\gamma_{1},\gamma_{2},II}_{p,q}\Rightarrow
\begin{array}{rl}
&\int^{2^{-2j'\beta}}_{0}\sum\limits_{(\epsilon'',k'')\in S^{w',j'}_{j,k}}|v^{\epsilon''}_{j',k''}(s)|^{p}ds\lesssim2^{-nj+p\gamma_{2}j}2^{-p(\gamma_{1}+\frac{n}{2}-\frac{n}{p})j'}2^{-2j'\beta}.
\end{array}
$$
By (\ref{eqn:est3}), H\"older's inequality and Lemma \ref{le7} (ii) we get
$$\begin{array}{rl}
|b^{\epsilon,4}_{j,k}(t)|
&\lesssim 2^{-\frac{nj}{2}+j+\frac{nj}{p}+(p-1)\gamma_{2}j}\sum\limits_{j<j'+2}\sum\limits_{w\in\mathbb{Z}^{n}}(1+|w|)^{-N}
2^{-(\frac{n}{2}+\gamma_{1})j'}2^{\frac{(p-1)nj'}{p}}\\
&\quad 2^{(2-p)\gamma_{2}j'}2^{-\frac{2\beta(p-1)j'}{p}}
\Big(\int^{2^{-2j'\beta}}_{0}\sum\limits_{(\epsilon',k')\in S^{w,j'}_{j,k}}|u^{\epsilon'}_{j',k'}(s)|^{p}ds\Big)^{\frac{1}{p}}.
\end{array}$$
By (\ref{eqn:est5}), $s2^{2j'\beta}\leq 1$ and $m'<1$, we can similarly achieve that for $\delta>0$,
$$\begin{array}{rl}
II^{4}_{b,Q_{r}}(t)
\lesssim&|Q_{r}|^{\frac{q\gamma_{2}}{n}-\frac{q}{p}}
\sum\limits_{-\log_{2}r<j<-\frac{\log_{2}t}{2\beta}}2^{qj(2-2\beta+p\gamma_{2})}
\Big[\sum\limits_{w\in\mathbb{Z}^{n}}(1+|w|)^{-N}\sum\limits_{j<j'+2}2^{\delta(j'-j)}\\
&2^{-pj'(2-2\beta+p\gamma_{2})}
2^{pj'[\frac{n}{2}+\gamma_{1}-\frac{n}{p}]}
\int^{2^{-2j'\beta}}_{0}\sum\limits_{(\epsilon',k')\in S^{w,j'}_{r}}|u^{\epsilon'}_{j',k'}(s)|^{p}(s2^{2j'\beta})^{m'}\frac{ds}{s}\Big]^{\frac{q}{p}}.
\end{array}
$$

If $q\leq p$, we apply the $\alpha-$triangle
inequality to obtain
$$\begin{array}{rl}
II^{4}_{b, Q_{r}}(t)\lesssim&|Q_{r}|^{\frac{q\gamma_{2}}{n}-\frac{q}{p}}\sum\limits_{w\in\mathbb{Z}^{n}}(1+|w|)^{-\frac{qN}{p}}
\sum\limits_{-\log_{2}r<j<-\frac{\log_{2}t}{2\beta}}
\sum\limits_{j<j'+2}2^{q(j-j')(2-2\beta+p\gamma_{2}-\frac{\delta}{p})}\\
&2^{qj'[\frac{n}{2}+\gamma_{1}-\frac{n}{p}]}
\Big[\int^{2^{-2j'\beta}}_{0}\sum\limits_{(\epsilon',k')\in S^{w,j'}_{r}}|u^{\epsilon'}_{j',k'}(s)|^{p}(s2^{2j'\beta})^{m'}\frac{ds}{s}\Big]^{\frac{q}{p}}.
\end{array}$$
If
$0<\frac{\delta}{p}<2(1-\beta)+p\gamma_{2}$, then
$II^{4}_{b,Q_{r}}(t)\lesssim\|u\|_{\B^{\gamma_{1},\gamma_{2},IV}_{p,q,m'}}$
follows from changing the order of $j$ and $j'$.

If $q>p$, by H\"older's inequality we get
$$\begin{array}{rl}
II^{4}_{b, Q_{r}}(t)
\lesssim&\ \sum\limits_{w\in\mathbb{Z}^{n}}\frac{|Q_{r}|^{\frac{q\gamma_{2}}{n}-\frac{q}{p}}}{(1+|w|)^{N}}
\sum\limits_{j\geq-\log_{2}r}\Big(\sum\limits_{j<j'+2}2^{(j-j')[2p(1-\beta)+p^{2}\gamma_{2}-\delta]}\Big)^{\frac{q-p}{p}}\\
&\times\ \Big\{\sum\limits_{j<j'+2}2^{(j-j')[2p(1-\beta)+p^{2}\gamma_{2}-\delta]}2^{qj'(\frac{n}{2}+\gamma_{1}-\frac{n}{p})}\times\\
&\quad \Big[\int^{2^{-2j'\beta}}_{0}\sum\limits_{(\epsilon',k')\in S^{w,j'}_{r}}|u^{\epsilon'}_{j',k'}(s)|^{p}(s2^{2j'\beta})^{m'}\frac{ds}{s}\Big]^{\frac{q}{p}}\Big\}
\lesssim \|u\|_{\B^{\gamma_{1},\gamma_{2},IV}_{p,q,m'}}.
\end{array}$$

\subsection{The setting (iv)}\label{sec11}
Again, we only consider the situation $1<p\leq2$ in the sequel.  For $i=4,5$, define
$$\begin{array}{rl}
IV^{m',i}_{b,Q_{r}}=&\ |Q_{r}|^{\frac{q\gamma_{2}}{n}-\frac{q}{p}}\sum\limits_{j\geq-\log_{2}r}2^{qj(\gamma_{1}+\frac{n}{2}-\frac{n}{p})}
\Big[\int^{2^{-2j\beta}}_{0}\sum\limits_{(\epsilon,k)\in S^{j}_{r}}|b^{\epsilon,i}_{j,k}(t)|^{p}(t2^{2j\beta})^{m'}\frac{dt}{t}\Big]^{\frac{q}{p}}.
\end{array}$$
Due to their similarity, we only check the first one of the following two cases:
 \begin{equation}\nonumber
\begin{cases}
\text{ Case 8.9: }
(t,x)\mapsto\sum\limits_{(\epsilon,j,k)\in\Lambda_{n}}b^{\epsilon,4}_{j,k}(t)\Phi^{\epsilon}_{j,k}(x)\ \ \hbox{is\ in}\ \
\B^{\gamma_{1},\gamma_{2}, IV}_{p,q,m'};\\
\text{ Case 8.10: } (t,x)\mapsto\sum\limits_{(\epsilon,j,k)\in\Lambda_{n}}b^{\epsilon,5}_{j,k}(t)\Phi^{\epsilon}_{j,k}(x)\ \ \hbox{is\ in}\ \
\B^{\gamma_{1},\gamma_{2}, IV}_{p,q,m'}.\\
\end{cases}
\end{equation}

Because $0<s<2^{-2j'\beta}$, one gets
$|v^{\epsilon''}_{j',k''}(s)|\lesssim2^{-\frac{nj'}{2}}2^{-j'(\gamma_{1}-\gamma_{2})}$
and
$$\begin{array}{rl}
\int^{2^{-2j'\beta}}_{0}\sum\limits_{(\epsilon'',k'')\in S^{w',j'}_{j,k}}|v^{\epsilon''}_{j',k''}(s)|^{p}ds\lesssim2^{-nj+p\gamma_{2}j}2^{-p(\frac{n}{2}+\gamma_{1}-\frac{n}{p})j'}2^{-2j'\beta},
\end{array}$$
By using (\ref{eqn:est3}), Lemma \ref{le7} (ii) and H\"older's inequality we get
$$\begin{array}{rl}
|b^{\epsilon,4}_{j,k}(t)|
&\lesssim2^{-\frac{nj}{2}+j+\frac{nj}{p}+(p-1)\gamma_{2}j}\sum\limits_{w\in\mathbb{Z}^{n}}
(1+|w|)^{-N}\sum\limits_{j<j'+2}2^{-(\frac{n}{2}+\gamma_{1})j'}\\
&2^{\frac{nj'(p-1)}{p}+(2-p)\gamma_{2}j'-\frac{2j'\beta(p-1)}{p}}\Big(\int^{2^{-2j'\beta}}_{0}
\sum\limits_{(\epsilon',k')\in S^{w,j'}_{j,k}}|u^{\epsilon'}_{j',k'}(s)|^{p}ds\Big)^{\frac{1}{p}}.
\end{array}$$
For $t2^{2j\beta}<1$, we then get
$$\begin{array}{rl}
IV^{m',4}_{b,Q_{r}}
\lesssim&\ |Q_{r}|^{\frac{q\gamma_{2}}{n}-\frac{q}{p}}\sum\limits_{j\geq-\log_{2}r}2^{qj(\gamma_{1}+\frac{n}{2}-\frac{n}{p})}
\Big\{\int^{2^{-2j\beta}}_{0}\sum\limits_{(\epsilon,k)\in S^{j}_{r}}2^{-\frac{pnj}{2}+pj+nj+p(p-1)\gamma_{2}j}\\
&\quad\Big[\sum\limits_{w\in\mathbb{Z}^{n}}
(1+|w|)^{-N}\sum\limits_{j<j'+2}2^{-(\frac{n}{2}+\gamma_{1})j'}
2^{\frac{nj'(p-1)}{p}+(2-p)\gamma_{2}j'}\\
&\quad2^{-\frac{2j'\beta(p-1)}{p}}\Big(\int^{2^{-2j'\beta}}_{0}
\sum\limits_{(\epsilon',k')\in S^{w,j'}_{j,k}}|u^{\epsilon'}_{j',k'}(s)|^{p}ds\Big)^{\frac{1}{p}}\Big]^{p}(t2^{2j\beta})^{m'}\frac{dt}{t}\Big\}^{\frac{q}{p}}.
\end{array}$$
Furthermore, we first apply H\"older's inequality to $w$, and then employ (\ref{eqn:est5}) to get that for
$\delta>0$,
$$\begin{array}{rl}
IV^{m',4}_{b,Q_{r}}
\lesssim&|Q_{r}|^{\frac{q\gamma_{2}}{n}-\frac{q}{p}}\sum\limits_{j\geq-\log_{2}r}2^{qj(\gamma_{1}+1+(p-1)\gamma_{2})}2^{-\frac{q\delta
j}{p}}\\
&\Big\{\int^{2^{-2j\beta}}_{0}\sum\limits_{(\epsilon,k)\in S^{j}_{r}}
\sum\limits_{j<j'+2}2^{\delta j'}
2^{-p(\frac{n}{2}+\gamma_{1})j'}2^{(p-1)nj'+p(2-p)\gamma_{2}j'-2\beta(p-1)j'}\\
&\sum\limits_{w\in\mathbb{Z}^{n}}\frac{1}{(1+|w|)^{N}}\Big(\int^{2^{-2j'\beta}}_{0}\sum\limits_{(\epsilon',k')\in S^{w,j'}_{j,k}}|u^{\epsilon'}_{j',k'}(s)|^{p}ds\Big)(t2^{2j'\beta})^{m'}\frac{dt}{t}\Big\}^{\frac{q}{p}}.
\end{array}$$
Because
$\int^{2^{-2j\beta}}_{0}(t2^{2j\beta})^{m}\frac{dt}{t}\lesssim1$, we obtain
$$\begin{array}{rl}
IV^{m',4}_{b,Q_{r}}\lesssim&|Q_{r}|^{\frac{q\gamma_{2}}{n}-\frac{q}{p}}\sum\limits_{j\geq-\log_{2}r}2^{qj(\gamma_{1}+1+(p-1)\gamma_{2})}2^{-\frac{q\delta
j}{p}}\\
&\quad\Big[\sum\limits_{j<j'+2}2^{\delta j'}
2^{pj'(\gamma_{1}+\frac{n}{2}-\frac{n}{p})} 2^{2\beta
j'}2^{-pj'(2-2\beta+p\gamma_{2})}\\
&\quad\sum\limits_{w\in\mathbb{Z}^{n}}(1+|w|)^{-N}\Big(\int^{2^{-2j'\beta}}_{0}\sum\limits_{(\epsilon',k')\in S^{w,j'}_{r}}|u^{\epsilon'}_{j',k'}(s)|^{p}ds\Big)\Big]^{\frac{q}{p}}.
\end{array}$$

If $q\leq p$, noticing $p\gamma_{2}+2-2\beta>0$
and taking
$0<\delta<p(\gamma_{1}+1+(p-1)\gamma_{2})$ we obtain
$$\begin{array}{rl}
IV^{m',4}_{b,Q_{r}}
\lesssim&\!\!\!\sum\limits_{w\in\mathbb{Z}^{n}}|Q_{r}|^{\frac{q\gamma_{2}}{n}-\frac{q}{p}}(1+|w|)^{-\frac{qN}{p}}
\sum\limits_{j\geq-\log_{2}r}\sum\limits_{j<j'+2}2^{q(\frac{\delta}{p}-p\gamma_{2}-2+2\beta)(j'-j)}\\
&\quad\quad 2^{qj'(\gamma_{1}+\frac{n}{2}-\frac{n}{p})}
\Big[\int^{2^{-2j'\beta}}_{0}\sum\limits_{(\epsilon',k')\in S^{w,j'}_{r}}|u^{\epsilon'}_{j',k'}(s)|^{p}(s2^{2j'\beta})^{m'}\frac{ds}{s}\Big]^{\frac{q}{p}}\lesssim \|u\|_{\B^{\gamma_{1},\gamma_{2},
IV}_{p,q,m'}},
\end{array}$$
where we have actually used the inequality
$$(s2^{2j'\beta})^{1-m'}\leq1\ \ \hbox{for}\ \
s\leq2^{-2j'\beta}\ \ \&\ \ 1-m'>0.
$$
and then changed the order of $j$ and $j'$.

If $q>p$, by H\"older's inequality we have
$$\begin{array}{rl}
IV^{m',4}_{b,Q_{r}}
&\lesssim\ \sum\limits_{w\in\mathbb{Z}^{n}}|Q_{r}|^{\frac{q\gamma_{2}}{n}-\frac{q}{p}}(1+|w|)^{-N}
\sum\limits_{j\geq-\log_{2}r}
\Big[\sum\limits_{j<j'+2}2^{p(2-2\beta+p\gamma_{2}-\frac{\delta}{p})(j-j')}\\
&\quad2^{pj'(\gamma_{1}+\frac{n}{2}-\frac{n}{p})}2^{2\beta j'}
\Big(\int^{2^{-2j'\beta}}_{0}\sum\limits_{(\epsilon',k')\in S^{w,j'}_{r}}|u^{\epsilon'}_{j',k'}(s)|^{p}ds\Big)\Big]^{\frac{q}{p}}
\lesssim \|u\|_{\B^{\gamma_{1},\gamma_{2}, IV}_{p,q,m'}}.
\end{array}$$

\end{document}